\documentclass[12pt]{amsart}
\usepackage[utf8]{inputenc}
\usepackage[letterpaper,total={6.5in,9in},top=1in, left=1in, right=1in, bottom=1in]{geometry}
\usepackage{fancyhdr}
\usepackage{libertine}
\usepackage[libertine]{newtxmath}
\usepackage{color,amsmath,mathtools,graphicx}
\usepackage{bbold}
\usepackage{epsfig,fancyhdr}
\usepackage{dsfont} 
\usepackage{graphicx,amsmath,mathtools,float}
\usepackage{epsfig,color,fancyhdr,setspace}
\usepackage{dsfont}
\usepackage{epstopdf}
\usepackage{import}
\usepackage{lineno}
\usepackage{stackrel,dsfont}
\usepackage[allcolors = blue,colorlinks]{hyperref}
\usepackage[abs]{overpic}
\usepackage{bibunits}[alpha]
\usepackage[center]{caption}
\usepackage{subcaption}
\usepackage{tikz-cd}
\usepackage{enumitem}
\counterwithin{figure}{section}
\usepackage[utf8]{inputenc}
\usepackage[T1]{fontenc}
\usepackage{ragged2e}
\usepackage{extarrows}
\usepackage{rotating}
\usepackage{tikz}
\usepackage{pgfplots}
\pgfplotsset{compat=1.18}
\allowdisplaybreaks
\newtheorem{theorem}{Theorem}[section]
\newtheorem{lemma}[theorem]{Lemma}
\newtheorem{definition}[theorem]{Definition}

\newtheorem{proposition}[theorem]{Proposition}
\newtheorem{remark}[theorem]{Remark}
\newtheorem{corollary}[theorem]{Corollary}

\newtheorem{innercustomgeneric}{\customgenericname}
\providecommand{\customgenericname}{}
\newcommand{\newcustomtheorem}[2]{%
  \newenvironment{#1}[1]
  {%
   \renewcommand\customgenericname{#2}%
   \renewcommand\theinnercustomgeneric{##1}%
   \innercustomgeneric
  }
  {\endinnercustomgeneric}
}

\newcustomtheorem{customthm}{Theorem}
\newcustomtheorem{customlemma}{Lemma}

\title{Kauffman bracket skein module of two families of Seifert manifolds}

\author{Minyi Liang} 
\address{Jilin University, Changchun, China}
\email{\rm myliang24@mails.jlu.edu.cn}

\author{Shangjun Shi} 
\address{East China Normal University, Shanghai, China}
\email{\rm 51255500022@stu.ecnu.edu.cn}

\author{Xiao Wang} 
\address{Jilin University, Changchun, China}
\email{\rm wangxiaotop@jlu.edu.cn}
\begin{document}

\begin{bibunit}

\maketitle

\begin{abstract}
    We compute the Kauffman bracket skein modules of Seifert manifolds $\Sigma_{0,1}((k_1,1),(k_2,1))$ and $\Sigma_{0,0}((k_1,1),(k_2,1),(k_3,1))$  by providing presentations of them. From the obtained presentations, we show that the Kauffman bracket skein modules of $\Sigma_{0,1}((k_1,1),(k_2,1))$ are free with infinitely many generators when $k_1,k_2\geq 1$ and that of $\Sigma_{0,0}((k_1,1),(k_2,1),(k_3,1))$ are finitely generated when $k_1,k_2,k_3\geq 2$. We also show that the empty link in either case is not trivial.

\end{abstract}

\section{Introduction}\

    In the late of 1980's, the theory of skein modules, a $3$-manifold invariant was introduced independently by Przytycki\cite{prz1} and Turaev\cite{tur} to generalize the polynomial link invariants in $S^3$.
    Since their introduction, skein modules have played an important role across multiple disciplines. In quantum topology, they model TQFT state spaces \cite{BLANCHET1995883}. In algebraic geometry, they recover coordinate rings of $SL_2(\mathbb{C})$ character varieties \cite{Bullock1997RingsOS}\cite{PRZYTYCKIA2000115}. In hyperbolic geometry, skein algebras quantize decorated Teichmüller spaces and Poisson structures \cite{Turaev1991}\cite{Bonahon2010QuantumTF}. Finally, skein modules appear in quantum field theory as algebraic encodings of Wilson loop observables and hyperbolic volume approximations \cite{d701919c6c204ff59f5a847690b6694c}\cite{Kashaev:1996kc}.

    Despite its importance, the computation of skein modules over the Laurent polynomial ring $\mathbb{Z}[A^{\pm}]$ is in general a difficult task.
    In our previous work computing $S_{2,\infty}(\#^2 S^1\times S^2)$\cite{bakshi2024kauffmanbracketskeinmodule}, we found that the product of Chebyshev decorated curves in genus two handlebody serve as a good basis. In this article, we try to extend our technique to  certain Seifert manifolds obtained from genus two handlebody by attaching $2$- and $3$-handles.

    We present here the main theorem we obtained in this article.

    \begin{theorem}
    We compute the Kauffman bracket skein module of the following two kinds of Seifert manifolds.
\label{main thm}
\begin{enumerate}
    \item  When $k_1,k_2\geq 1$, the Kauffman bracket skein module of $\Sigma_{0,1}((k_1,1),(k_2,1))$ is an infinitely generated free module over $Z[A^{\pm}]$.

    \item  When, $k_1,k_2,k_3\geq 2$, the Kauffman bracket skein module of $\Sigma_{0,0}((k_1,1),(k_2,1),(k_3,1))$ is finitely generated over $Z[A^{\pm}]$, and its minimal number of generators is less than $(k_1+1)(k_2+1)(k_3+1)$.

\end{enumerate}
   
    \end{theorem}
    The finitely generatedness of small Seifert manifolds has been fully studied by Detcherry, Kalfagianni and Sikora in \cite{detcherry2024skeinmodulescharactervarieties}. We provide an elementary proof for certain cases, and we hope to find their explicit structure in future work.

   This article is organized as follows: In Section $2$, we  firstly recall some foundational definitions and properties of the Kauffman bracket skein module theory. Then, we recall the Seifert manifolds and their surgery diagrams. In Section $3$, we determine the relation submodule of $S_{2,\infty}(D^2(k_1,k_2))$ and $S_{2,\infty}(S^2(k_1,k_2,k_3))$ through their surgery diagrams. In Section $4$, for $k_1,k_2\geq 1$, we show that $S_{2,\infty}(D^2(k_1,k_2))$ is an infinitely generated free module and provide an explicit generating set. In Section $5$, we prove that $S_{2,\infty}(S^2(k_1,k_2,k_3))$ is finitely generated.
\section{Basic definitions and properties}

In this section, we recall the definition of the Kauffman bracket skein module and its properties. Also, we introduce the Heegaard diagrams and the surgery diagrams of the Seifert manifolds $\Sigma_{0,0}((k_1,1),(k_2,1),(k_3,1))$ and $\Sigma_{0,1}((k_1,1),(k_2,1))$.

\subsection{Kauffman bracket skein module}
\begin{definition}
    Let $M$ be an oriented 3-manifold, $R$ a commutative ring with unity, and $A\in R$ a fixed invertible element, the Kauffman bracket skein module of $M$, denoted by $S_{2,\infty}(M;R,A)$, is the quotient of the free $R$-module generated by ambient isotopy classes of unoriented framed links (including the empty link $\emptyset$) in $M$, by the Kauffman bracket skein relations, which are 
    $$(1)\,L_+-AL_0-A^{-1}L_{\infty};\quad (2)\,L\sqcup \bigcirc +(A^2+A^{-2})L.$$
    Where $\bigcirc$ denotes the trivial framed knot in $M$ and the skein triple $(L_+,L_0,L_{\infty})$ denotes three framed links in $M$, which are identical except in a small 3-ball in $M$ where they differ as illustrated in Figure \ref{Skein triple for the Kauffman bracket skein module}.
    
    \begin{figure}[H]
    \centering
    \includegraphics[width=0.7\linewidth]{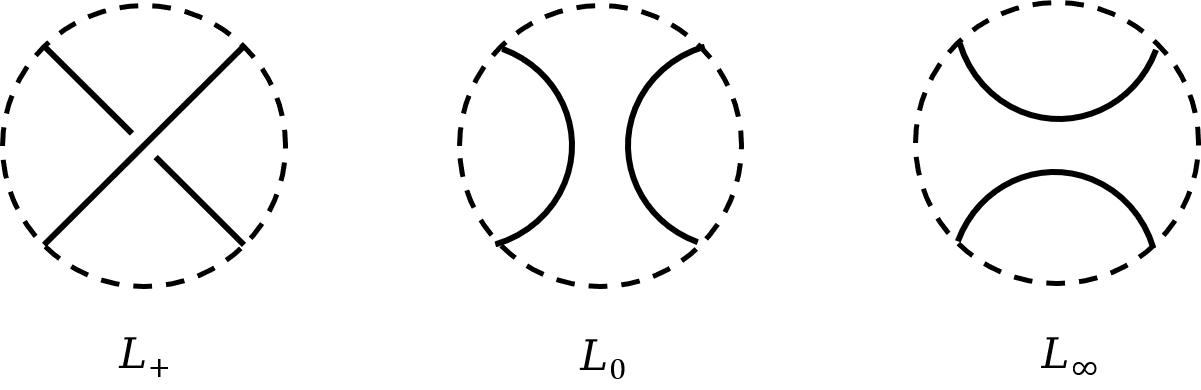}
    \caption{Skein triple for the Kauffman bracket skein module.}
        \label{Skein triple for the Kauffman bracket skein module}
    \end{figure}
\end{definition}
\begin{remark}
    We denote $S_{2,\infty}(M;\mathbb{Z}[A^{\pm 1}],A)$ by $S_{2,\infty}(M)$ for convenience.
\end{remark}
It is also convenient for us to use the relative Kauffman bracket skein module, and here we provide its definition.

\begin{definition}
    Let $M$ be an oriented 3-manifold with boundary, $R$ a commutative ring with unity, and $A\in R$ a fixed invertible element, $\left\{x_i\right\}_1^{2n}\subset \partial M$ be framed points(embedding of small intervals), the relative Kauffman bracket skein modules of $(M,\left\{x_i\right\}_1^{2n})$, denoted by $S_{2,\infty}(M,\left\{x_i\right\}_1^{2n};R,A)$, is the quotient of the free $R$-module generated by ambient isotopy classes of unoriented framed links (including the empty link $\emptyset$) in $(M,\left\{x_i\right\}_1^{2n})$ keeping $\left\{x_i\right\}_1^{2n}$ fixed, by the Kauffman relations as in Figure \ref{Skein triple for the Kauffman bracket skein module}.
\end{definition}

The following theorem tells how the skein module of a $3$-dimensional manifold changes when attaching a $2$-handle or a $3$-handle, which is fundamental for our computation.
\begin{theorem}[\cite{1998Fundamentals}]
    \ \\
    \begin{enumerate}
    \item[(1)] Let $i: M \hookrightarrow N$ be an orientation preserving embedding of 3-manifolds. This yields a module homomorphism  $i_{*}: S_{2, \infty}(M;R,A) \longrightarrow S_{2, \infty}(N;R,A)$.\\
     \item[(2)]Let $M=M_{1} \sqcup M_{2}$ be the disjoint union of oriented 3-manifolds $M_{1}$ and $M_{2}$. Then $S_{2, \infty}(M;R,A) \cong S_{2, \infty}\left(M_{1};R,A\right) \otimes_{R} S_{2, \infty}\left(M_{2};R,A\right).$\\
     \item[(3)]If $N$ is obtained from $M$ by adding a 3-handle to M and $i: M \hookrightarrow N$ is the associated embedding, then the induced homomorphism $i_{*}: S_{2, \infty}(M;R,A) \longrightarrow S_{2, \infty}(N;R,A)$ is an isomorphism.\\
     \item[(4)] (Handle Sliding Lemma) Let $(M, \partial M)$ be a 3-manifold with boundary and $\gamma$ be a simple closed curve on  $\partial M$. Additionally, let $N=M_{\gamma}$ be the 3-manifold obtained from M by adding a 2-handle along $\gamma$ and $i: M \hookrightarrow N$ be the associated embedding. Then the induced homomorphism $i_{*}: S_{2, \infty}(M;R,A) \longrightarrow S_{2, \infty}(N;R,A)$ is an epimorphism. Furthermore, the kernel of $i_{*}$ is generated by the relations yielded by 2-handle slidings. In particular, if $\mathcal{L}_{\text {gen }}^{f r}$ is a set of framed links in $M$ that generates $\mathcal{S}_{2, \infty}(M;R,A)$, then  $S_{2, \infty}(N;R,A) \cong S_{2, \infty}(M;R,A) / \mathcal{J}$, where $\mathcal{J}$ is the submodule of $S_{2, \infty}(M;R,A)$ generated by the expressions $L-s l_{\gamma}(L)$. Here $L \in \mathcal{L}_{\text {gen }}^{f r}$ and $s l_{\gamma}(L)$ is obtained from L by sliding it along $\gamma$.
    \end{enumerate}
\end{theorem}

\begin{remark}

In this article, we will call $\mathcal{J}$ the handle sliding submodule of $N$ from $M$ and abbreviate $M$ when it is clear from context.
    
\end{remark}

The handle sliding lemma can be generalised to the case where a manifold is obtained by attaching more than one 2-handle to the 3-manifold $M$, which leads to the following corollary.
\begin{corollary}
    
\label{2.2}\cite{McLENDON2006Detecting}
    Let $M,\partial M$ be a 3-manifold with boundary and $\beta_1$, $\beta_2$,...,$\beta_n$ be disjoint simple closed curves in $\partial M$. Glue n 2-handles to $M$ along the curves $\beta_i$ and denote the resultant 3-manifold by $N$. If $\mathcal{J}_i$ is the submodule of $S_{2,\infty}(M;R,A)$ generated by handle slides along $\beta_i$, then $S_{2,\infty}(N;R,A)\cong S_{2,\infty}(M;R,A)/(\mathcal{J}_1+\mathcal{J}_2+\cdots+\mathcal{J}_n)$.
\end{corollary}

\begin{theorem}\label{2.3}\cite{Bullock_2005,BLP}
    Consider any relative curve $\alpha$ in $(H_n;\{u,v\})$. Now, handle slidings in $(H_n)_\beta$ take place locally in the neighbourhood of the curve $\beta$. For every relative curve $\alpha$, handle sliding in $(H_n)_\beta$ replaces the curve $\alpha\cup\beta_2$ with the curve $\alpha\cup\beta_1$. This gives the handle sliding relation, $\alpha\cup\beta_2\equiv\alpha\cup\beta_1$. By introducing the $R$-linear homomorphism $\omega$: $S_{2,\infty}(H_n,\{u,v\};R,A)\to S_{2,\infty}(H_n;R,A)$, defined by $\omega(\alpha)=\alpha\cup\beta_2-\alpha\cup\beta_1$,  we see that $\omega(S_{2,\infty}(H_n,\{u,v\};R,A))=\mathcal{J}$. where $\mathcal{J}$ is the submodule of $S_{2,\infty}(H_n;R,A)$ generated by the expressions $L-s l_{\beta}(L)$, $L \in \mathcal{L}_{\text {gen}}^{fr}$ of $H_n$.
\end{theorem}

Notice that, based on Corollary \ref{2.2} and Theorem \ref{2.3}, one can always provide a infinite presented presentation of $S_{2,\infty}(N;R,A)$ with $N$ obtained by attaching $2$-handles on $H_n$. The main goal of our work is to understand the relation submodule and dig out the structure of the module itself.

\subsection{Skein module of $I$-bundles over surfaces}
Next, we recall some properties of the Kauffman skein modules of $I$-bundles over surfaces.
\begin{theorem}\cite{2006Skein}\label{2006skein}
    Let $M$ be an oriented 3-manifold which is equal to $F\times I$, where $F$ is an oriented surface. Then $S_{2,\infty}(M;R,A)$, is a free module with a basis $B(F)$ consisting of isotopy class of links in F without contractible components (but including the empty link).
\end{theorem}

\begin{remark}\cite{2006Skein}\label{relative S}
    The same applies to the case of relative Kauffman bracket skein module, where $\left\{x_i\right\}_1^{2n}\subset \partial F\times I$, these fixed framed points are projected onto different endpoints on $\partial F$. For simplicity, we denote $S_{2,\infty}(F\times I;R,A)$ by $S_{2,\infty}(F;R,A)$ and $S_{2,\infty}(F\times I,\left\{x_i\right\}_1^{2n};R,A)$ by $S_{2,\infty}(F,\left\{x_i\right\}_1^{2n};R,A)$.
\end{remark}
In this article, we denote $\Sigma_{0,3}\times I$ by $H_2$, boundary of $\Sigma_{0,3}$ by $a_1\cup a_2\cup a_3$, and two fixed framed points by $u,v$, where $u\in a_1\times I$, $v\in a_3\times I$ (see Figure \ref{_2_3_SSS_SC}).

\begin{definition}
    $S_q$ denotes the Chebyshev polynomials of the second kind, which satisfies the recurrence relation $S_{q+1}(x) =xS_q(x)-S_{q-1}(x)$, with the initial conditions $S_0(x) = 1$ and $S_1(x) = x$. When $n$ is negative, we define $S_{n}(x) = -S_{-n-2}(x)$ for convenience. When $\alpha$ is an element is our skein module, we call the element $S_{q}(\alpha)$ the $q$-th Chebyshev decoration of $\alpha$.

\end{definition}

\begin{proposition}
    $S_{2,\infty}(H_2;R,A)$ is a free module with basis  $\left\{S_{l_1}(a_1)S_{l_2}(a_2)S_{l_3}(a_3)\right\}_{l_i\ge 0}$, where $a_1$, $a_2$, and $a_3$ represent the homotopically nontrivial curves on $\Sigma_{0,3}$ as illustrated in Figure \ref{_2_3_SSS_SC}. The empty link is represented by $S_0(a_1)S_0(a_2)S_0(a_3)$.
    \begin{figure}[H]
    \centering
    \includegraphics[width=1\linewidth]{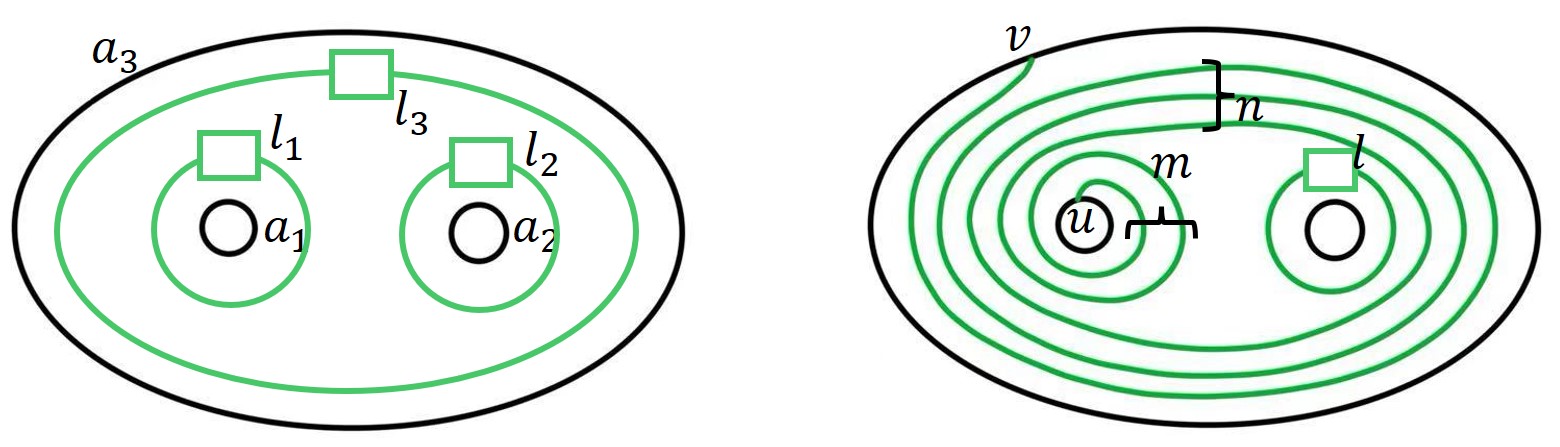}
    \caption{$S_{l_1}(a_1)S_{l_2}(a_2)S_{l_3}(a_3)$ and $S_{l}(a_2)C_{m,n}$}
        \label{_2_3_SSS_SC}
    \end{figure}
    \begin{proof}
    It follows from Theorem \ref{2006skein} since $H_2=\Sigma_{0,3}\times I$.
        
    \end{proof}
\end{proposition}

\begin{proposition}
     $S_{2,\infty}(H_2,\{u,v\};R,A)$ is a free module with basis $\left\{S_l(a_2)C_{m,n}\right\}_{l\ge 0}$, as shown in Figure \ref{_2_3_SSS_SC}.
    \begin{proof}
        It follows from Remark \ref{relative S}.
    \end{proof}
\end{proposition}

\subsection{Seifert manifolds and its surgery diagram}
Now, we recall the definition of Seifert manifolds.
\begin{definition}
    A Seifert manifold (also a Seifert fibered space) is a closed 3-manifold which can be decomposed into a disjoint union of $S^1$’s (called fibers), such that each tubular neighbourhood of a fiber is a standard fibered torus.
\end{definition}

    We use the notation $\Sigma_{ g,b}((\alpha_1,\beta_1),\cdots,(\alpha_k,\beta_k))$ for this Seifert manifold $M$, where $\Sigma_{ g,b}$ is a orientable compact surface, $g$ is the genus, and $b$ is the number of boundary components. Seifert manifolds can be constructed as follows. Consider a compact connected surface $B$. Let $D_1,\cdots,D_k$ be disjoint disks in the interior of $B$, and $B'$ be the surface obtained from $B$ by removing interior of those disks. Let $M'\to B'$ be the circle bundle with $M'$ orientable. Let $s:B'\to M'$ be a cross section of $M'\to B'$. We can choose a diffeomorphism $\phi$ between each component of $\partial M'$ and a copy of  $S^1\times S^1$ by taking the cross section to $S^1\times\left\{y\right\}$ (slope 0) and a fiber to $\left\{x\right\}\times S^1$ (slope $\infty$). We are assuming the standard fact that each nontrivial circle in $S^1\times S^1$ is isotopic to a unique ‘linear’ circle which lifts to the line $y=(p/q)x$ of slope $p/q$ in the universal cover $\mathbb{R}^2$. From $M'$ we construct a manifold $M$ by attaching $k$ solid tori $D^2\times S^1$ to the torus components $T_i$ of $\partial M'$ lying over $\partial D_i\subset \partial B'$. The attachment is given by diffeomorphisms taking a meridian circle $\partial D^2\times \left\{y\right\}$ of $\partial D^2\times S^1$ to a circle of some finite slope $\beta_i/\alpha_i\in \mathbb{Q}$ in $T_i$. The $k$ slopes $\alpha_i/\beta_i$ determine $M$ uniquely, since once the meridian disk $D^2\times \left\{y\right\}$ is attached to $M'$ there is only one way to fill in a ball to complete the attaching of $D^2\times S^1$. \par When $g=b=0$ and $k=3$, $M$ is called a small Seifert manifold.
    
\begin{remark}
    In this paper, for convenience, we denote
    $D^2(k_1,k_2)=\Sigma_{0,1}((k_1,1),(k_2,1))$, and $S^2(k_1,k_2,k_3)=\Sigma_{0,0}((k_1,1),(k_2,1),(k_3,1))$, $k_i\in \mathbb{Z}$.
\end{remark}

A small Seifert manifold $M$ has two models as shown in Figure \ref{Two models of $S^2(k_1,k_2,k_3)$}. On the left we have the vertical Heegaard splitting of $M$ \cite{Moriah1998IrreducibleHS}, and on the right we have its surgery diagram \cite{akbulut20164}. In Figure \ref{_2_7_Heegaard_Surface_in_Kirby_Diagram}, the Heegaard surface of $M$ is shown, and the blue curve corresponds to the blue curve in 
 the Heegaard splitting in Figure \ref{Two models of $S^2(k_1,k_2,k_3)$}.

\begin{figure}[H]
        \centering
        \includegraphics[width=1\linewidth]{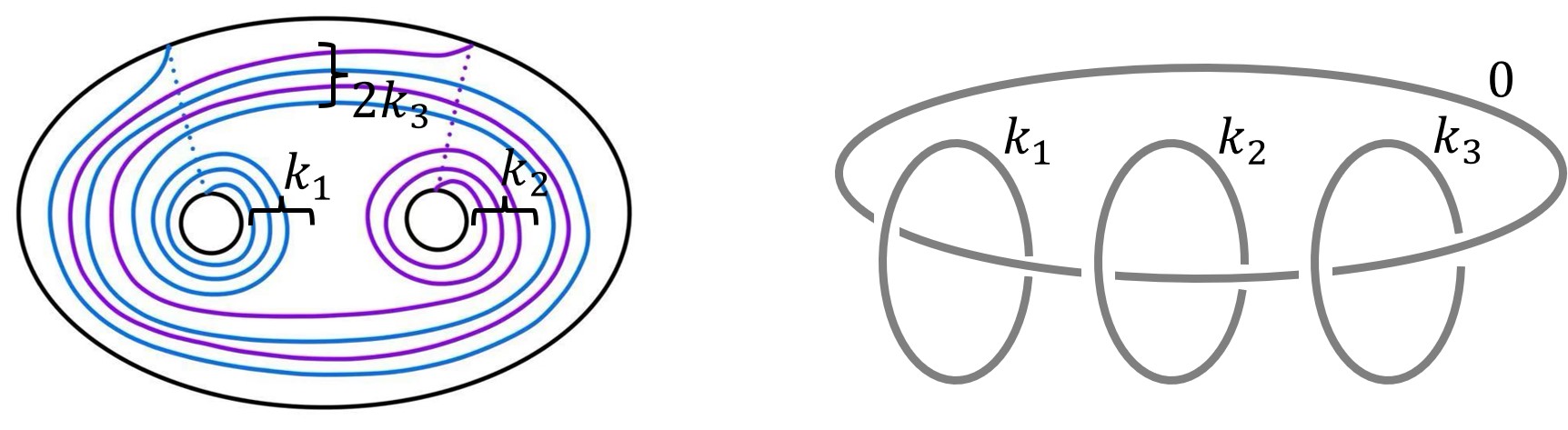}
        \caption{Heegaard splitting and surgery diagram of $S^2(k_1,k_2,k_3)$}
        \label{Two models of $S^2(k_1,k_2,k_3)$}
\end{figure}

\begin{figure}[H]
        \centering
        \includegraphics[width=0.5\linewidth]{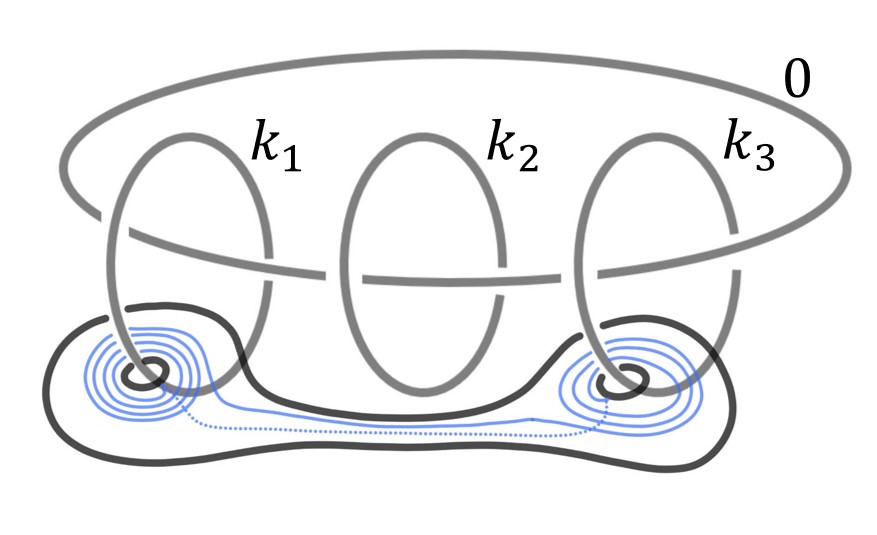}
        \caption{Heegaard surface in surgery diagram of $S^2(k_1,k_2,k_3)$}
        \label{_2_7_Heegaard_Surface_in_Kirby_Diagram}
\end{figure}

\section{Handle sliding submodules of $D^2(k_1,k_2)$ and $S^2(k_1,k_2,k_3)$}
In this section, we will use the surgery diagram to analyze the relation submodules in $S_{2,\infty}(D^2(k_1,k_2))$ and $S_{2,\infty}(S^2(k_1,k_2,k_3))$. Eventually, we will express them in terms of products of Chebyshev decorated boundary parallel curves.
\subsection{The model $R(n_1,n_2)$}
In this subsection, we introduce the model framed link $R(n_1,n_2)$ and study its properties.  Elements in relation submodules of $S_{2,\infty}(D^2(k_1,k_2))$ and $S_{2,\infty}(S^2(k_1,k_2,k_3))$ will share similar properties with $R(n_1,n_2)$.
\begin{definition}
    Let $R(n_1,n_2)$ be a framed link, which is represented by the white bands along the black edges in Figure \ref{Definition of R(n_1,n_2)} and the gray parts indicate the link in $S^3$ where surgeries with integer coefficients $k_i$ are performed. We can envision the construction of  $R(n_1,n_2)$ as follows: Firstly, suppose there exists a plane $\alpha$ (depicted in yellow), and on each side of this plane, the outermost intersection lines of these solid tori with plane $\alpha$ are $\gamma_1$ and $\gamma_2$, respectively, which we designate as longitudes. Secondly, assume that there are two framed links lying on the surfaces of these solid tori, parallel to their boundaries, with slopes $n_1$ and $n_2$ ,intersecting $\gamma_1$ and $\gamma_2$ at a single point, respectively. Then, attach a 1-handle at the intersection points and allow the framed links on the solid tori to connect along the boundaries of this 1-handle, thus constructing $R(n_1,n_2)$.
    \begin{figure}[H]
    \centering
    \includegraphics[width=1\linewidth]{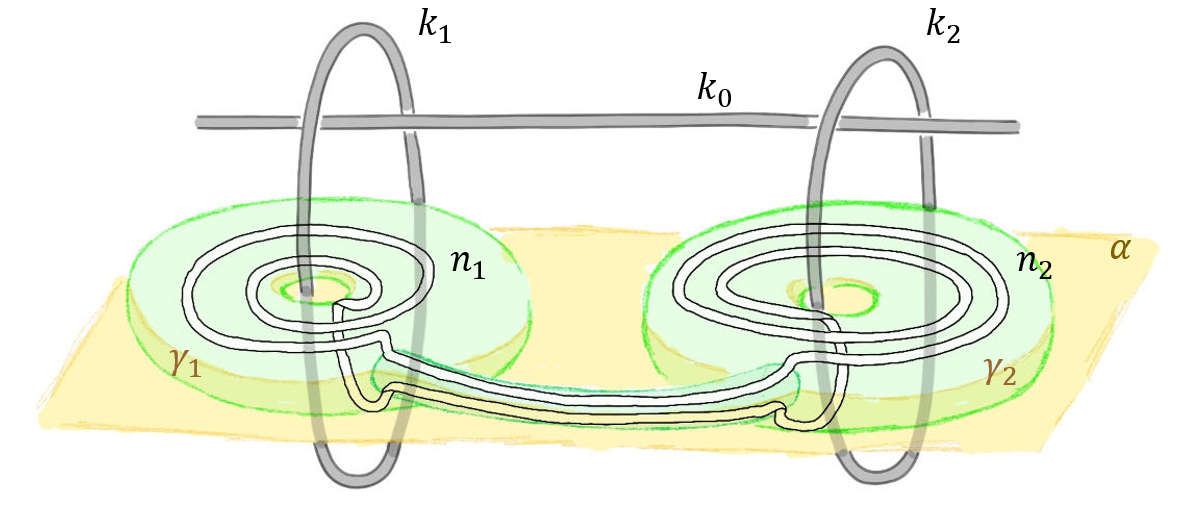}
    \caption{$R(n_1,n_2)$(in this figure, $n_1<0$, $n_2>0$)}
        \label{Definition of R(n_1,n_2)}
    \end{figure}
\end{definition}

\begin{proposition}
    $R(n_1,n_2)=R(-n_1+k_1,-n_2+k_2)$
    \begin{proof}
       As is illustrated in Figure \ref{Proof of R R}.
        \begin{figure}[H]
        \centering
        \includegraphics[width=1\linewidth]{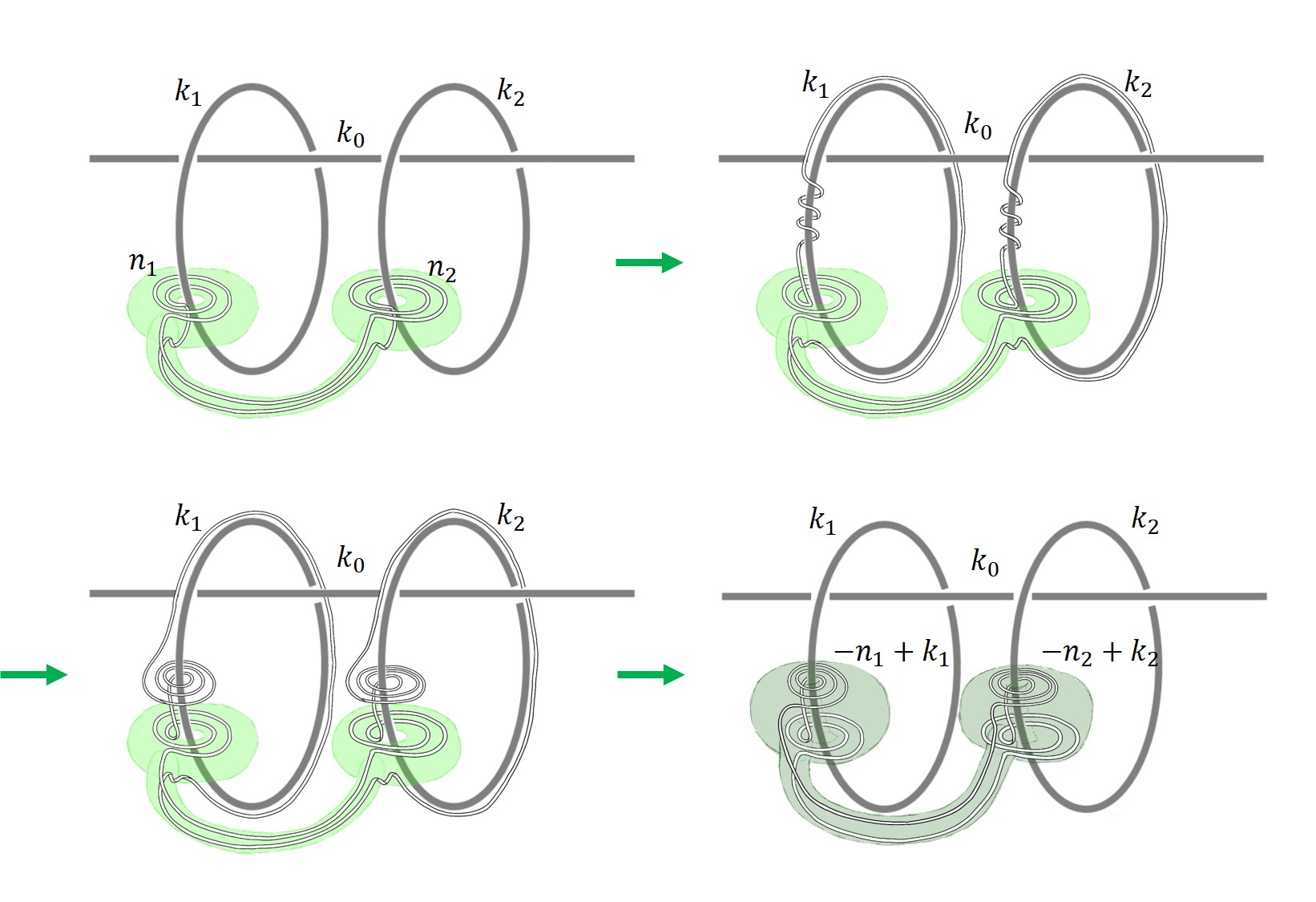}
        \caption{$R(n_1,n_2)=R(-n_1+k_1,-n_2+k_2)$}
        \label{Proof of R R}
        \end{figure}
    \end{proof}
\end{proposition}

For convenience, we denote the elements on the left-hand side of Figure \ref{a1 a2 b S} as $a_1$, $a_2$ and $b$, respectively, and the elements on the right-hand side as $S_{l_1}(a_1)S_{l_2}(a_2)R(n_1,n_2)$.
\begin{figure}[H]
        \centering
        \includegraphics[width=1\linewidth]{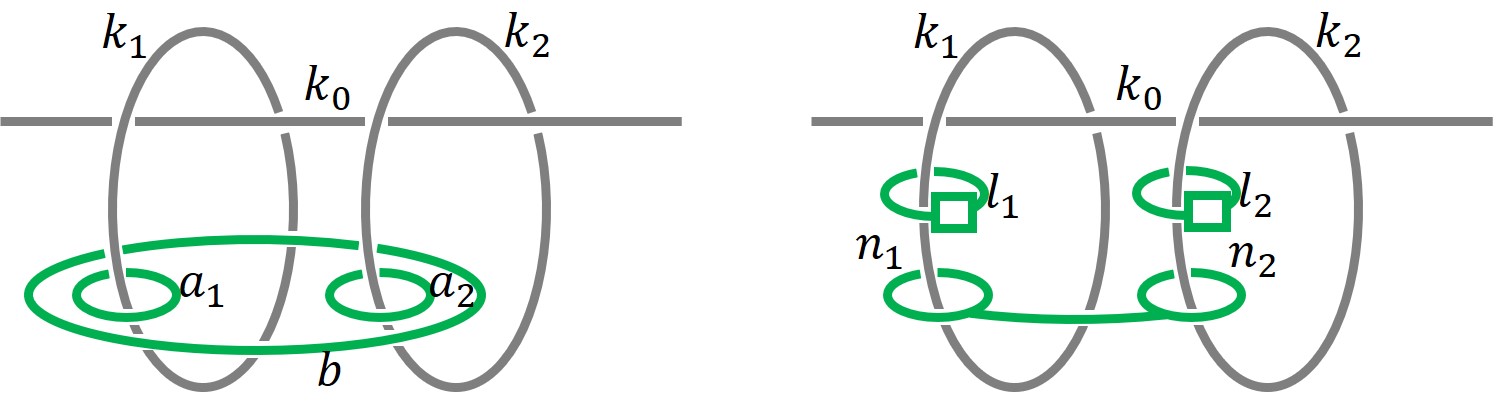}
        \caption{$a_1$, $a_2$, $b$ (left) and $S_{l_1}(a_1)S_{l_2}(a_2)R(n_1,n_2)$ (right)}
        \label{a1 a2 b S}
\end{figure}

\begin{lemma}
    \begin{align}
        &R(n_1,n_2)=A^{- l_1}S_{l_1}(a_1)R(n_1- l_1,n_2)-A^{- l_1- 1}S_{l_1-1}(a_1)R(n_1- l_1- 1,n_2),\nonumber\\&R(n_1,n_2)=A^{+ l_1}S_{l_1}(a_1)R(n_1+l_1,n_2)-A^{+ l_1+ 1}S_{l_1-1}(a_1)R(n_1+ l_1+ 1,n_2),\nonumber\\&R(n_1,n_2)=A^{- l_2}S_{l_2}(a_2)R(n_1,n_2- l_2)-A^{- l_2- 1}S_{l_2- 1}(a_2)R(n_1,n_2-l_2- 1),\nonumber\\&R(n_1,n_2)=A^{+ l_2}S_{l_2}(a_2)R(n_1,n_2+ l_2)-A^{+ l_2+ 1}S_{l_2- 1}(a_2)R(n_1,n_2+l_2+ 1),\nonumber
    \end{align}
    where $l_i\ge 0$.
    \begin{proof}
        As shown in Figure \ref{Skein_relation_of_R}, we have
        \begin{align}
            &R(n_1,n_2)=A^{-1}a_1R(n_1-1,n_2)-A^{-2}R(n_1-2,n_2),\nonumber\\
            &R(n_1,n_2)=A^{+1}a_1R(n_1+1,n_2)-A^{+2}R(n_1+2,n_2).\nonumber
        \end{align}
        Therefore,
        \begin{align}
        R(n_1,n_2) & = A^{\pm 1}a_1R(n_1\pm 1,n_2)-A^{-2}R(n_1\pm 2,n_2)\nonumber\\
         & = A^{\pm 1}a_1\left [A^{\pm 1}a_1R(n_1\pm 2,n_2)-A^{\pm 2}R(n_1\pm 3,n_2)\right ] -A^{\pm 2}R(n_1\pm 2,n_2)\nonumber\\
         & = A^{\pm 2}S_2(a_1)R(n_1\pm 2,n_2)-A^{\pm 3}S_1(a_1)R(n_1\pm 3,n_2)\nonumber\\
         & = A^{\pm 3}S_3(a_1)R(n_1\pm 3,n_2)-A^{\pm 4}S_2(a_1)R(n_1\pm 4,n_2)\nonumber\\
         & = ......\nonumber\\
         & =A^{\pm l_1}S_{l_1}(a_1)R(n_1\pm l_1,n_2)-A^{\pm l_1\pm 1}S_{l_1-1}(a_1)R(n_1\pm l_1\pm 1,n_2).\nonumber
         \end{align}
        Other equations follows in a similar manner.
        \begin{figure}[H]
        \centering
        \includegraphics[width=1\linewidth]{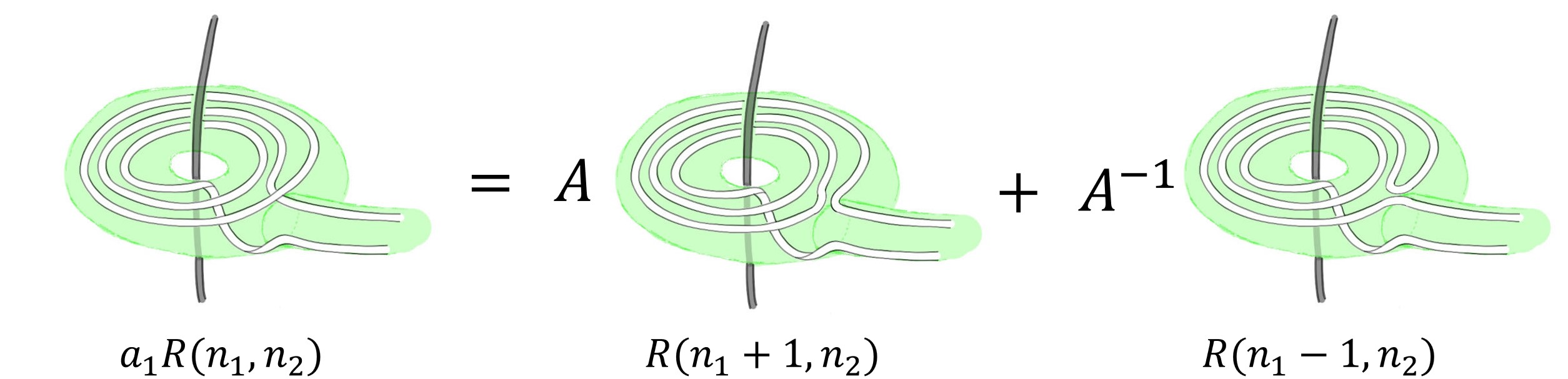}
        \caption{Skein relation of $R(n_1,n_2)$}
        \label{Skein_relation_of_R}
        \end{figure}
    \end{proof}
    \label{Lemma 3.1}
\end{lemma}
We have the following calculation.
\begin{lemma}
    \begin{align}
        &R(\pm 1,0)=-A^{\mp 3}a_1,&&R(0,\pm 1)=-A^{\mp 3}a_2,\nonumber\\&R(0,0)=\bigcirc=-A^2-A^{-2},&&R(\pm 1,\mp 1)=b.\nonumber
    \end{align}
    \label{Lemma 3.2}
\end{lemma}

\begin{proposition}
    \begin{align}
        R(n_1,n_2)=&-A^{-n_1-n_2-2}S_{n_1}(a_1)S_{n_2}(a_2)-A^{-n_1-n_2+2}S_{n_1-2}(a_1)S_{n_2-2}(a_2)\nonumber\\&-A^{-n_1-n_2}S_{n_1-1}(a_1)S_{n_2-1}(a_2)b.\nonumber
    \end{align}
    This equation is illustrated in Figure \ref{_3_5_Proposition_3_2}.
    \begin{figure}[H]
        \centering
        \includegraphics[width=1\linewidth]{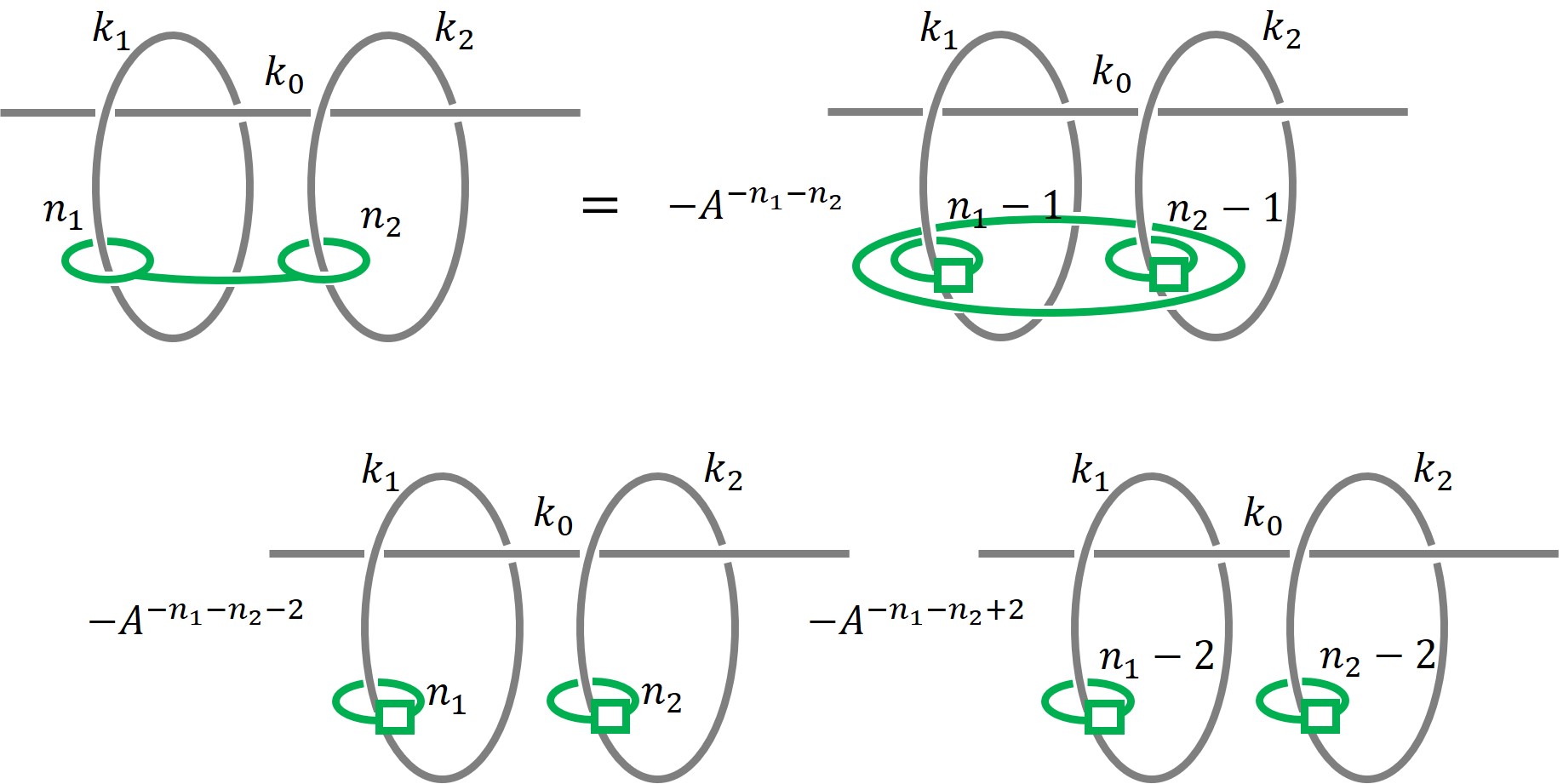}
        \caption{Calculation of $R(n_1,n_2)$}
        \label{_3_5_Proposition_3_2}
    \end{figure}
    \begin{proof}
    \ \\
    \begin{enumerate}
        \item $n_1=n_2=0$, it is obivious.
        \item $n_1>0,n_2\ge 0$,
        \begin{align}
            R(n_1,n_2)\xlongequal{Lemma\ \ref{Lemma 3.1}}&A^{-n_1+1}S_{n_1-1}(a_1)R(1,n_2)-A^{-n_1}S_{n_1-2}(a_1)R(0,n_2)\nonumber\\\xlongequal{Lemma\ \ref{Lemma 3.1}}&A^{-n_1+1}S_{n_1-1}(a_1)\left[A^{-n_2}S_{n_2}(a_2)R(1,0)-A^{-n_2-1}S_{n_2-1}(a_2)R(1,-1)\right]\nonumber\\&-A^{-n_1}S_{n_1-2}(a_1)\left[A^{-n_2}S_{n_2}(a_2)R(0,0)-A^{-n_2-1}S_{n_2-1}(a_2)R(0,-1)\right]\nonumber\\\xlongequal{Lemma\ \ref{Lemma 3.2}}&-A^{-n_1-n_2-2}S_{n_1-1}(a_1)S_{n_2}(a_2)a_1-A^{-n_1-n_2+2}S_{n_1-2}(a_1)S_{n_2-1}(a_2)a_2\nonumber\\&+A^{-n_1-n_2-2}S_{n_1-2}(a_1)S_{n_2}(a_2)+A^{-n_1-n_2+2}S_{n_1-2}(a_1)S_{n_2}(a_2)\nonumber\\&-A^{-n_1-n_2}S_{n_1-1}(a_1)S_{n_2-1}(a_2)b\nonumber\\=&-A^{-n_1-n_2-2}S_{n_1}(a_1)S_{n_2}(a_2)-A^{-n_1-n_2+2}S_{n_1-2}(a_1)S_{n_2-2}(a_2)\nonumber\\&-A^{-n_1-n_2}S_{n_1-1}(a_1)S_{n_2-1}(a_2)b.\nonumber
        \end{align}
        Other cases follow similarly.
    \end{enumerate}
     
    \end{proof}
    \label{Proposition 3.2}
\end{proposition}

\subsection{Relators on $S_{2,\infty}(D^2(k_1,k_2))$ and $S_{2,\infty}(S^2(k_1,k_2,k_3))$}\ \\

First, we introduce some notations.
\begin{definition}
    $R_{13}(n_1,n_2,n_3)$, $R_{23}(n_1,n_2,n_3)$ and $R_{12}(n_1,n_2,n_3)$ are as shown in the Figure \ref{_4_3_Def_of_R}.
    \begin{figure}[H]
        \centering
        \includegraphics[width=0.7\linewidth]{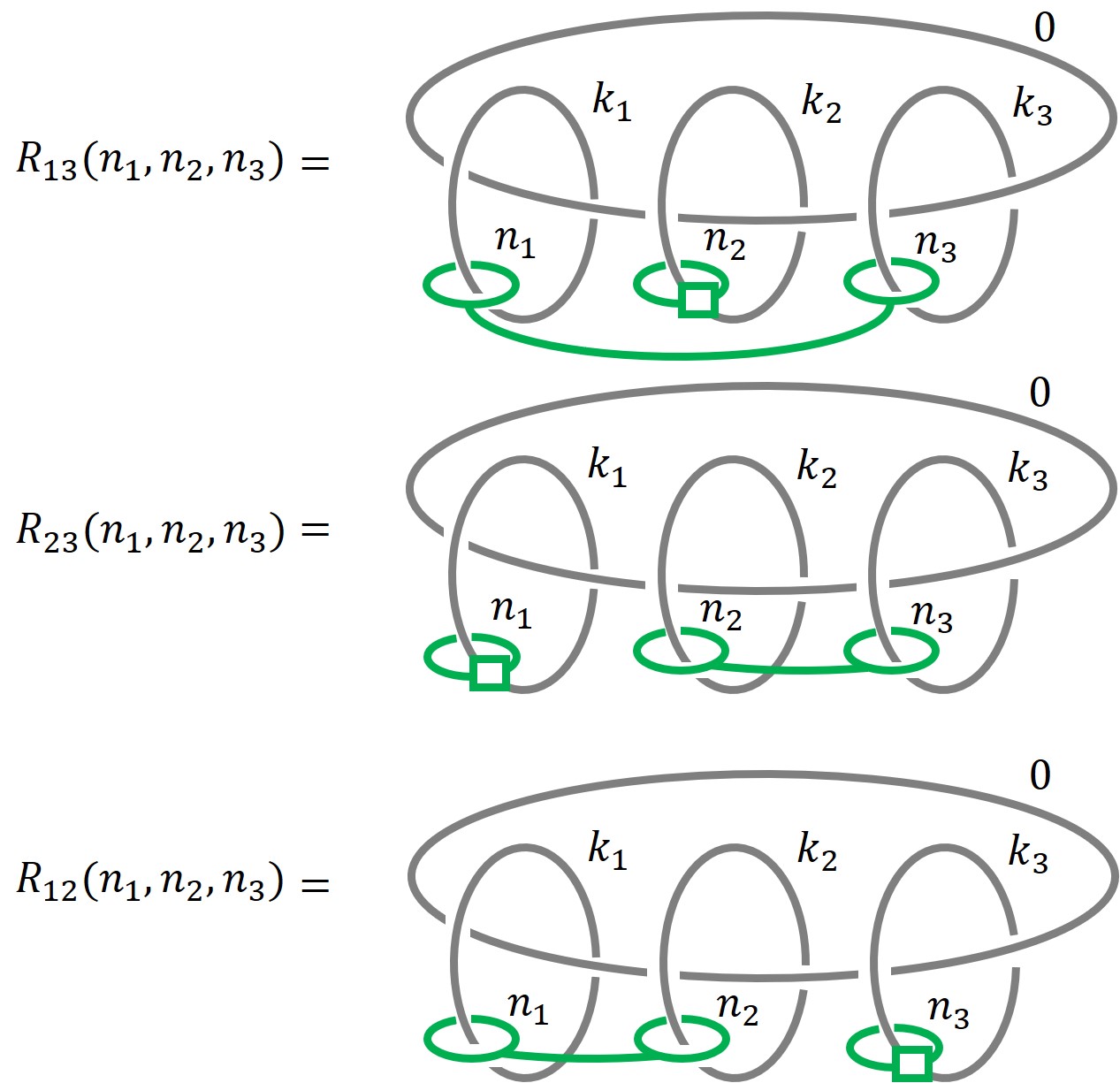}
        \caption{$R_{13}(n_1,n_2,n_3)$, $R_{23}(n_1,n_2,n_3)$ and $R_{12}(n_1,n_2,n_3)$}
        \label{_4_3_Def_of_R}
    \end{figure}
\end{definition}

\begin{remark}
    We abbreviate $S_{n}(a_i)$ by $s_i^{n}$ for convenient.
\end{remark}

\begin{lemma}
    \begin{align}
        R_{13}(n_1,n_2,n_3)=&-A^{-n_1-n_3-2}s_1^{n_1}s_2^{n_2}s_3^{n_3}-A^{-n_1-n_3+2}s_1^{n_1-2}s_2^{n_2}s_3^{n_3-2}\nonumber\\&-A^{-n_1-n_3}s_1^{n_1-1}s_2^{n_2+1}s_3^{n_3-1}-A^{-n_1-n_3}s_1^{n_1-1}s_2^{n_2-1}s_3^{n_3-1},\nonumber\\
        R_{23}(n_1,n_2,n_3)=&-A^{-n_2-n_3-2}s_1^{n_1}s_2^{n_2}s_3^{n_3}-A^{-n_2-n_3+2}s_1^{n_1}s_2^{n_2-2}s_3^{n_3-2}\nonumber\\&-A^{-n_2-n_3}s_1^{n_1+1}s_2^{n_2-1}s_3^{n_3-1}-A^{-n_2-n_3}s_1^{n_1-1}s_2^{n_2-1}s_3^{n_3-1},\nonumber\\R_{12}(n_1,n_2,n_3)=&-A^{-n_1-n_2-2}s_1^{n_1}s_2^{n_2}s_3^{n_3}-A^{-n_1-n_2+2}s_1^{n_1-2}s_2^{n_2-2}s_3^{n_3}\nonumber\\&-A^{-n_1-n_2}s_1^{n_1-1}s_2^{n_2-1}s_3^{n_3+1}-A^{-n_1-n_2}s_1^{n_1-1}s_2^{n_2-1}s_3^{n_3-1}.\nonumber
    \end{align}
    \begin{proof}
        \begin{align}
            R_{12}(n_1,n_2,n_3)\xlongequal{Proposition\ \ref{Proposition 3.2}}&-A^{-n_1-n_2-2}s_1^{n_1}s_2^{n_2}s_3^{n_3}-A^{-n_1-n_2+2}s_1^{n_1-2}s_2^{n_2-2}s_3^{n_3}\nonumber\\&-A^{-n_1-n_2}s_1^{n_1-1}s_2^{n_2-1}s_3^{n_3}b\nonumber\\\xlongequal{b=s_3}&-A^{-n_1-n_2-2}s_1^{n_1}s_2^{n_2}s_3^{n_3}-A^{-n_1-n_2+2}s_1^{n_1-2}s_2^{n_2-2}s_3^{n_3}\nonumber\\&-A^{-n_1-n_2}s_1^{n_1-1}s_2^{n_2-1}s_3^{n_3+1}-A^{-n_1-n_2}s_1^{n_1-1}s_2^{n_2-1}s_3^{n_3-1}.\nonumber
        \end{align}
        For other cases, the proof is similar.
    \end{proof}
\end{lemma}

\begin{theorem}
    $S_{2,\infty}(S^2(k_1,k_2,k_3))=S_{2,\infty}(H_2)/\mathcal{J}_{13}+\mathcal{J}_{23}$, where $S_{2,\infty}(H_2)$ is a free module generated by $\left \{ s_1^{l_1}s_2^{l_2}s_3^{l_3}\right\}_{l_i\ge0 }$, $\mathcal{J}_{13}$ and $\mathcal{J}_{23}$ are submodules of $S_{2,\infty}(H_2)$ generated by $\left \{R_{13}^{n_1,n_2,n_3}\right\}_{n_i\in \mathbb{Z}}$ and $\left \{R_{23}^{n_1,n_2,n_3}\right\}_{n_i\in \mathbb{Z}}$ respectively. Where
    \begin{align}
        &R_{13}^{n_1,n_2,n_3}=R_{13}(n_1,n_2,n_3)-R_{13}(-n_1+k_1,n_2,-n_3+k_3),\nonumber\\
        &R_{23}^{n_1,n_2,n_3}=R_{23}(n_1,n_2,n_3)-R_{23}(n_1,-n_2+k_2,-n_3+k_3).\nonumber
    \end{align}
    \begin{proof}
        The Heegaard surface of this manifold is shown in Figure \ref{_2_7_Heegaard_Surface_in_Kirby_Diagram}. We only consider the case of handle sliding by the blue curve. From Figure \ref{_4_4_Handle_Sliding}, we observe the changes of this curve after sliding along $\beta_1\cup \beta_2$. Each sliding is equivalent to the transformation in the surgery diagram as stated in Proposition 3.1. Furthermore, due to Proposition 2.2, if we choose $S_{n_2}(a_2)C_{n_1,n_3}$ as shown in the Figure \ref{_2_3_SSS_SC}, $S_{n_2}(a_2)C_{n_1,n_3}$ is a basis. Notice that  $S_{n_2}(a_2)C_{n_1,n_3}\cup \beta_1=R_{13}(n_1,n_2,n_3)$ and $S_{n_2}(a_2)C_{n_1,n_3}\cup \beta_2=R_{13}(-n_1+k_1,n_2,-n_3+k_3)$. We have $\omega(S_{n_2}(a_2)C_{n_1,n_3})=R_{13}(n_1,n_2,n_3)-R_{13}(-n_1+k_1,n_2,-n_3+k_3)$, which we denote by $R_{13}^{n_1,n_2,n_3}$. $\{R_{13}^{n_1,n_2,n_3}\}$ generate submodule $\mathcal{J}_{13}$ according to Theorem 2.3. Finally, by  Theorem 2.2, we have the desired result.
    \end{proof}

    \begin{figure}[H]
        \centering
        \includegraphics[width=1\linewidth]{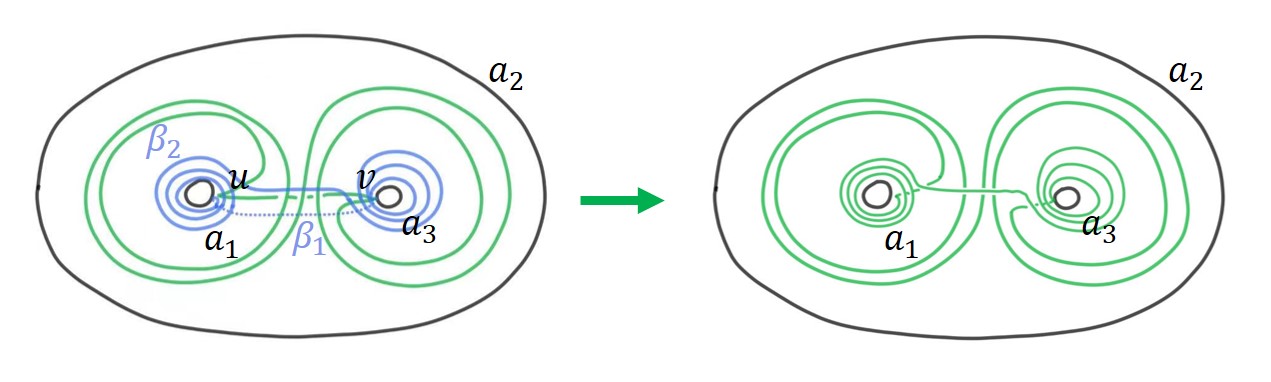}
        \caption{Change after handle sliding}
        \label{_4_4_Handle_Sliding}
    \end{figure}
    \label{Theorem 4.1}
\end{theorem}

\begin{corollary}\label{D}
    $S_{2,\infty}(D^2(k_1,k_2))=S_{2,\infty}(H_2)/\mathcal{J}_{12}$, where $S_{2,\infty}(H_2)$ is a free module generated by $\left \{ s_1^{l_1}s_2^{l_2}s_3^{l_3}\right\}_{l_i\ge0 }$, $\mathcal{J}_{12}$ is submodule of $S_{2,\infty}(H_2)$ generated by $\left \{R_{12}^{n_1,n_2,n_3}\right\}_{n_i\in \mathbb{Z}}$. Where
    \begin{align}
        &R_{12}^{n_1,n_2,n_3}=R_{12}(n_1,n_2,n_3)-R_{12}(-n_1+k_1,-n_2+k_2,n_3).\nonumber
    \end{align}
\end{corollary}
\begin{remark}
    $\mathcal{J}_{13}+\mathcal{J}_{23}=\mathcal{J}_{12}+\mathcal{J}_{23}=\mathcal{J}_{12}+\mathcal{J}_{13}=\mathcal{J}_{12}+\mathcal{J}_{13}+\mathcal{J}_{23}$.
\end{remark}
\subsection{Relations among relators}
In this subsection, we list some equations for later reducing the superfluous relators. In the following lemmas, all relators are in $S_{2,\infty}(S^2(k_1,k_2,k_3))$.

\begin{lemma}
\label{Lemma formula 0}
    $$R_{12}^{-1,k_2+1,n_3}+A^{2}R_{12}^{0,k_2,n_3-1}+A^2R_{12}^{0,k_2,n_3+1}+A^{4}R_{12}^{1,k_2-1,n_3}=0.$$
    \begin{proof}
        Appendix.
    \end{proof}
\end{lemma}

\begin{lemma}
\label{Lemma formula 01}
    \begin{align}
        &R_{12}^{-n_1,k_2+n_2,n_3}-A^{2n_1+2n_2}R_{12}^{-n_1+k_1,n_2,n_3}\nonumber\\&+A^{2n_1-2}\sum_{i=0}^{n_1-2}\sum_{j=0}^{i}(-1)^iR_{12}^{n_1-2-i,k_2+n_2-i,n_3-i+2j}+A^{2n_1}\sum_{i=0}^{n_1-1}\sum_{j=0}^{i+1}(-1)^iR_{12}^{n_1-1-i,k_2+n_2-1-i,n_3-1-i+2j}\nonumber\\&+A^{2n_1}\sum_{i=1}^{n_1-2}\sum_{j=0}^{i-1}(-1)^iR_{12}^{n_1-1-i,k_2+n_2-1-i,n_3+1-i+2j}+A^{2n_1+2}\sum_{i=0}^{n_1-1}\sum_{j=0}^{i}(-1)^iR_{12}^{n_1-i,k_2+n_2-2-i,n_3-i+2j}\nonumber\\&-A^{2n_2+2}\sum_{i=0}^{n_1-2}\sum_{j=0}^{i}(-1)^iR_{12}^{n_1+k_1-2-i,n_2-i,n_3-i+2j}-A^{2n_2}\sum_{i=0}^{n_1-1}\sum_{j=0}^{i+1}(-1)^iR_{12}^{n_1+k_1-1-i,n_2-1-i,n_3-1-i+2j}\nonumber \\ &-A^{2n_2}\sum_{i=1}^{n_1-2}\sum_{j=0}^{i-1}(-1)^iR_{12}^{n_1+k_1-1-i,n_2-1-i,n_3+1-i+2j}-A^{2n_2-2}\sum_{i=0}^{n_1-1}\sum_{j=0}^{i}(-1)^iR_{12}^{n_1+k_1-i,n_2-2-i,n_3-i+2j}=0,\nonumber
    \end{align}
    where $n_1\ge 1$.
    \begin{proof}
        See Appendix.
    \end{proof}
\end{lemma}

\begin{lemma}
\label{Lemma formula 02}
     \begin{align}
        &R_{12}^{-n_1,k_2+n_2,n_3}-A^{2n_1+2n_2}R_{12}^{-n_1+k_1,n_2,n_3}\nonumber\\&+A^{2n_1+2}\sum_{i=0}^{n_2-2}\sum_{j=0}^{i}(-1)^iR_{12}^{n_1-i,k_2+n_2-2-i,n_3-i+2j}+A^{2n_1}\sum_{i=0}^{n_2-1}\sum_{j=0}^{i+1}(-1)^iR_{12}^{n_1-1-i,k_2+n_2-1-i,n_3-1-i+2j}\nonumber\\&+A^{2n_1}\sum_{i=1}^{n_2-2}\sum_{j=0}^{i-1}(-1)^iR_{12}^{n_1-1-i,k_2+n_2-1-i,n_3+1-i+2j}+A^{2n_1-2}\sum_{i=0}^{n_2-1}\sum_{j=0}^{i}(-1)^iR_{12}^{n_1-2-i,k_2+n_2-i,n_3-i+2j}\nonumber\\&-A^{2n_2-2}\sum_{i=0}^{n_2-2}\sum_{j=0}^{i}(-1)^iR_{12}^{n_1+k_1-i,n_2-2-i,n_3-i+2j}-A^{2n_2}\sum_{i=0}^{n_2-1}\sum_{j=0}^{i+1}(-1)^iR_{12}^{n_1+k_1-1-i,n_2-1-i,n_3-1-i+2j}\nonumber \\ &-A^{2n_2}\sum_{i=1}^{n_2-2}\sum_{j=0}^{i-1}(-1)^iR_{12}^{n_1+k_1-1-i,n_2-1-i,n_3+1-i+2j}-A^{2n_2+2}\sum_{i=0}^{n_2-1}\sum_{j=0}^{i}(-1)^iR_{12}^{n_1+k_1-2-i,n_2-i,n_3-i+2j}=0.\nonumber
    \end{align}
    where $n_2\ge 1$.
    \begin{proof}
        The proof is similar to that of Lemma \ref{Lemma formula 01}.
    \end{proof}
\end{lemma}

\begin{lemma}
\label{Lemma formula 1}
    When $n_1\ge 0$,
    \begin{align}
    &R_{23}^{n_1,n_2,n_3}-A^{n_1-n_3}\sum_{i=0}^{n_1}(-1)^i\left(R_{12}^{n_1-i,n_2-i,n_3+i}+A^2R_{12}^{n_1+1-i,n_2-1-i,n_3-1+i}\right)\nonumber\\&-A^{n_1+n_2-k_2}\sum_{i=0}^{n_1}(-1)^i\left(R_{13}^{n_1-i,n_2-k_2-2-i,-n_3+k_3-i}+A^2R_{13}^{n_1+1-i,n_2-k_2-1-i,-n_3+k_3-1-i}\right)=0, \nonumber\\
    &R_{23}^{n_1,n_2,n_3}-A^{n_1-n_2}\sum_{i=0}^{n_1}(-1)^i\left(R_{13}^{n_1-i,n_2+i,n_3-i}+A^2R_{13}^{n_1+1-i,n_2-1+i,n_3-1-i}\right)\nonumber\\&-A^{n_1+n_3-k_3}\sum_{i=0}^{n_1}(-1)^i\left(R_{12}^{n_1-i,-n_2+k_2-i,n_3-k_3-2-i}+A^2R_{12}^{n_1+1-i,-n_2+k_2-1-i,n_3-k_3-1-i}\right)=0. \nonumber
    \end{align}
    \begin{proof}
    See    Appendix.
    \end{proof}
\end{lemma}

\section{The modules $S_{2,\infty}(D^2(k_1,k_2))$, $k_i\ge1$}
In this section, we will show that $S_{2,\infty}(D^2(k_1,k_2))$ is free when $k_i\ge1$, and provide an explicit generating set. Since the roles of $k_1$ and $k_2$ are totally symmetric, without loss of generality, we prove the case when $k_2\geq k_1\geq 1$.
\subsection{Reduction of relators}
According to Corollary \ref{D}, the relation submodule of $S_{2,\infty}(D^2(k_1,k_2))$ is generated by $\mathcal{J}^{12}=\left \{R_{12}^{n_1,n_2,n_3}\right\}_{n_i\in \mathbb{Z}}$. By the properties developed in Section 3.3, we are able to reduce the generating set of relators, which is helpful for later discussion.
\begin{definition}

Let $J_{I_1,I_2}^{12}=\left\{R_{12}^{n_1,n_2,n_3}|n_1\in I_1,n_2\in I_2,n_3\ge 0\right\}$, $\mathcal{J}_{I_1,I_2}^{12}$ is the submodule of $S_{2,\infty}(H_2)$ generated by $J_{I_1,I_2}^{12}$.
\end{definition}

\begin{proposition}\label{Jmin}
    When $k_2\ge k_1\ge 1$, $\mathcal{J}^{12}$ is generated by $J^{12}_{min}$, where 
    \begin{align}
        J^{12}_{min}=
        J^{12}_{(k_1,),[0,)}\cup J^{12}_{(\frac{k_1}{2},k_1],(\frac{k_2}{2},)}\cup J^{12}_{[0,\frac{k_1}{2}],[\frac{k_2}{2},)}.\nonumber
    \end{align}
    \end{proposition}
\begin{figure}
    \centering

    \begin{tikzpicture}
  \begin{axis}[
    xlabel={$n_1$},
    ylabel={$n_2$},
    axis lines=middle,
    xmin=0, xmax=15,
    ymin=0, ymax=15,
    xtick={0,3,6,10},
    ytick={0,5},
    xticklabels={$0$, $\frac{k_1}{2}$, $k_1$, },
    yticklabels={$0$, $\frac{k_2}{2}$},
    enlargelimits,
    legend style={at={(1.05,1)},anchor=north west}
  ]

  % Region 1: {k1 < n1, n2 >= 0}
  \addplot [
    domain=6:15,
    samples=100,
    fill=blue!20,
    draw=none
  ] {15} \closedcycle;

  % Region 2: {k1/2 < n1 <= k1, n2 > k2/2}
  \addplot [
    fill=red!30,
    draw=none
] coordinates {
    (3,5) (6,5) (6,15) (3,15)
} \closedcycle;

  % Region 3: {0 <= n1 <= k1/2, n2 >= k2/2}
\addplot [
    fill=green!30,
    draw=none
] coordinates {
    (0,5) (3,5) (3,15) (0,15)
} \closedcycle;

  % Reference lines
  \addplot[blue,dashed] coordinates {(6,0)(6,5)}; % n1 = k1
    \addplot[red,dashed] coordinates {(3,5)(6,5)}; 
  \addplot[dashed] coordinates {(3,0)(3,5)}; 
  \addplot[green, thick] coordinates {(0,5)(3,5)};

  \end{axis}

\end{tikzpicture}
 \caption{$J^{12}_{min}$}
    \label{J}
\end{figure}

     We need Lemma \ref{lemma1}, \ref{lemma2} and \ref{lemma3} for the proof of Proposition \ref{Jmin}.
    \begin{remark}
        For convenience, we write $(k,)=(k,\infty)$, $[k,)=[k,\infty)$, $()=(-\infty,\infty)$ and so on.
    \end{remark}

        \begin{lemma}\label{lemma1}
        $\mathcal{J}^{12}=\mathcal{J}^{12}_{(),\left[\frac{k_2}{2},\right)}$.
            \begin{proof}
            By definitions, $\mathcal{J}^{12}_{(),\left[\frac{k_2}{2},\right)}\subset\mathcal{J}^{12}$, so we only need to show $\mathcal{J}^{12}\subset\mathcal{J}^{12}_{(),\left[\frac{k_2}{2},\right)}$. By definition(see Corollary \ref{D}), we have $$R_{12}^{n_1,n_2,n_3}+R_{12}^{n_1,n_2,-n_3-2}=0,$$ $$R_{12}^{n_1,n_2,n_3}+R_{12}^{-n_1+k_1,-n_2+k_2,n_3}=0.$$  The upper equation tells us that $\mathcal{J}^{12}$ is generated by $R_{12}^{n_1,n_2,n_3}$ with $n_3 \ge -1$. When $n_3=-1$, $R_{12}^{n_1,n_2,-1}=0$ by definition, which concludes $\mathcal{J}^{12}\subset\mathcal{J}^{12}_{(),()}$. The lower equation tells us that the relators symmetric about $(\frac{k_1}{2},\frac{k_2}{2})$ are linearly dependent, which concludes $\mathcal{J}^{12}_{(),\left(,\frac{k_2}{2}\right)}\subset\mathcal{J}^{12}_{(),\left[\frac{k_2}{2},\right)}$ 
               . Thus $\mathcal{J}^{12}\subset\mathcal{J}^{12}_{(),()}=\mathcal{J}^{12}_{(),\left[\frac{k_2}{2},\right)}$.

            \end{proof}
        \end{lemma}
        \begin{lemma}\label{lemma2}
            $\mathcal{J}^{12}_{[-n,),[-1,)}\subset\mathcal{J}^{12}_{[-n+1,),[-1,)},n\ge 1$.
            \begin{proof}
            $\forall R_{12}^{-n,k_2+n_2,n_3}\in \mathcal{J}^{12}_{[-n,),[-1,)}$, we divide our prove into three cases: 
            \begin{enumerate}
                \item when $n_2\ge n$, $ R_{12}^{-n,k_2+n_2,n_3}\in \mathcal{J}^{12}_{[-n+1,),[-1,)}$ as it can be directly checked to be represented by linear compositions of elements in $\mathcal{J}^{12}_{[-n+1,),[-1,)}$ by Lemma \ref{Lemma formula 01}; 
                \item when $n>n_2\ge 1$, $ R_{12}^{-n,k_2+n_2,n_3}\in \mathcal{J}^{12}_{[-n+1,),[-1,)}$ as it can be directly checked to be represented by linear compositions of elements in $\mathcal{J}^{12}_{[-n+1,),[-1,)}$ by Lemma \ref{Lemma formula 02};
                \item when $n_2\leq 0$, $ R_{12}^{-n,k_2+n_2,n_3}=-R_{12}^{n+k_1,-n_2,n_3}\in \mathcal{J}^{12}_{[-n+1,),[-1,)}$.
                Then we show that $\forall R_{12}^{-n,k_2+n_2,n_3}\in \mathcal{J}^{12}_{[-n,),[-1,)}$,$\forall R_{12}^{-n,k_2+n_2,n_3}\in \mathcal{J}^{12}_{[-n+1,),[-1,)}$.
            \end{enumerate}

            \end{proof}
        \end{lemma}
        \begin{lemma}\label{lemma3}
            $\mathcal{J}^{12}_{[-1,),[0)}=\mathcal{J}^{12}_{[0,),[-1,)}=\mathcal{J}^{12}_{[0,),[0,)}$.
            \begin{proof}
            Firstly, we will show $\mathcal{J}^{12}_{[-1,),[0,)}=\mathcal{J}^{12}_{[0,),[0,)}$. Clearly, $\mathcal{J}^{12}_{[0,),[0,)}\subset\mathcal{J}^{12}_{[-1,),[0,)}$, so we only need to show $\mathcal{J}^{12}_{[-1,),[0,)}\subset\mathcal{J}^{12}_{[0,),[0,)}$. $\forall R_{12}^{-1,k_2+n_2,n_3}\in \mathcal{J}^{12}_{[-1,),[0,)}$, we have 
                \begin{enumerate}
                    \item when $n_2=1$, $ R_{12}^{-1,k_2+n_2,n_3}\in \mathcal{J}^{12}_{[0,),[0,)}$ by Lemma \ref{Lemma formula 0}; 
                    \item  when $n_2>1$,$ R_{12}^{-1,k_2+n_2,n_3}\in\mathcal{J}^{12}_{[0,),[0,)}$ by Lemma \ref{Lemma formula 01},
                    \item when $n_2<1$, $ R_{12}^{-n,k_2+n_2,n_3}=-R_{12}^{n+k_1,-n_2,n_3}\in \mathcal{J}^{12}_{[0,),[0,)}$.
                \end{enumerate}
                Secondly, we will show $\mathcal{J}^{12}_{[0,),[-1,)}=\mathcal{J}^{12}_{[0,),[0,)}$,we only need to show $\mathcal{J}^{12}_{[0,),[-1,)}\subset\mathcal{J}^{12}_{[0,),[0,)}$ for a similar reason. $\forall R_{12}^{n_1+k_1,-1,n_3}\in \mathcal{J}^{12}_{[0,),[-1,)}$,we have\begin{enumerate}
       \item when $n_1=1$,$R_{12}^{n_1+k_1,-1,n_3}\in \mathcal{J}^{12}_{[0,),[0,)}$ by  the equation $R_{12}^{n_1,n_2,n_3}=-R_{12}^{-n_1+k_1,-n_2+k_2,n_3}$ and Lemma \ref{Lemma formula 0}, where we take $R_{12}^{-1,k_2,n_3}$ as $-R_{12}^{1+k_1,-1,n_3}$.
       \item when $n_1>1$,$R_{12}^{n_1+k_1,-1,n_3}\in \mathcal{J}^{12}_{[0,),[0,)}$ by  the equation $R_{12}^{n_1,n_2,n_3}=-R_{12}^{-n_1+k_1,-n_2+k_2,n_3}$ and Lemma \ref{Lemma formula 02}, where we take $R_{12}^{-n_1,n_2+k_2,n_3}$ as $-R_{12}^{n_1+k_1,-1,n_3}$
       \item when $n_1<1$,$R_{12}^{n_1+k_1,-1,n_3}=-R_{12}^{-n_1,1+k_2,n_3}\in\mathcal{J}^{12}_{[0,),[0,)}$.
   \end{enumerate}
             
            \end{proof}
        \end{lemma}
        
  \begin{proof}[\textbf{The proof of Proposition \ref{Jmin}}]
            
        The above lemmas combined together tell us that $\mathcal{J}^{12}$ is a submodule of $S_{2,\infty}(H_2)$ generated by $J^{12}_{[0,),[0,)}$:\begin{enumerate}
            \item  By Lemma  \ref{lemma1}, $\mathcal{J}^{12}_{(),\left[\frac{k_2}{2},\right)}=\mathcal{J}^{12}$. Also we have $\mathcal{J}^{12}_{(),\left[\frac{k_2}{2},\right)}\subset\mathcal{J}^{12}_{(),[-1,)}\subset\mathcal{J}^{12}$ by definition, which implies $\mathcal{J}^{12}=\mathcal{J}^{12}_{(),[-1,)}$.
            \item$\mathcal{J}^{12}_{(),[-1,)}=\mathcal{J}^{12}_{[0,),[-1,)}$by Lemma \ref{lemma2}.
            \item $\mathcal{J}^{12}_{[0,),[-1,)}=\mathcal{J}^{12}_{[0,),[0,)}$ by Lemma \ref{lemma3}. 
        \end{enumerate}
       Next, we will show that $\mathcal{J}^{12}_{[0,),[0,)}$ is generated by $J^{12}_{min}$. $\forall R_{12}^{n_1,n_2,n_3}\in J^{12}_{\left[0,\frac{k_1}{2}\right),[0,\frac{k_2}{2})}$,  $R_{12}^{n_1,n_2,n_3}=-R_{12}^{-n_1+k_1,-n_2+k_2,n_3}\in J^{12}_{(\frac{k_1}{2},k_1],(\frac{k_2}{2},k_2]}$, $\forall R_{12}^{n_1,n_2,n_3}\in J^{12}_{\left[\frac{k_1}{2},k_1\right],\left[0,\frac{k_2}{2}\right]}$, $R_{12}^{n_1,n_2,n_3}=-R_{12}^{-n_1+k_1,-n_2+k_2,n_3}\in J^{12}_{\left[0,\frac{k_1}{2}\right],\left[\frac{k_2}{2},k_2\right]}$.Intuitively, the remaining part of Figure \ref{J} can be symmetrically sent to the colored part through point $(\frac{k_1}{2},\frac{k_2}{2})$, $R_{12}^{\frac{k_1}{2},\frac{k_2}{2},n_3}=0$ when $
        \frac{k_1}{2},\frac{k_2}{2}\in \mathbb{Z}$. Therefore, the proposition is correct when $k_2\ge k_1\ge 1$.
        \end{proof}
\subsection{$S_{2,\infty}(D^2(k_1,k_2))$, $k_i\ge 1$}
Now, we are ready to give a finer presentation of $S_{2,\infty}(D^2(k_1,k_2))$, when $k_i\ge 1$.
\begin{lemma}
\label{P base chain}
    Let $\left\{G_n\right\}_{n\ge 0}$ be a family set that satisfies the condition $G_n\subset G_{n+1}$, and $G_{\infty}=\cup_{i=0}^{\infty}G_i$. Let $\mathcal{G}_n$ and $\mathcal{G}_{\infty}$ are free modules generated by $G_n$ and $G_{\infty}$ over $\mathbb{Z}[A^{\pm1}]$. $\left\{J_n\right\}_{n\ge 0}$ is a family of elements in $\mathcal{G}_{\infty}$, satisfying $J_n\subset J_{n+1}$, and $J_{\infty}=\cup_{i=0}^{\infty}J_i$. Let $\mathcal{J}_n$ and $\mathcal{J}_{\infty}$ are submodules of $\mathcal{G}_{\infty}$ generated by $J_n$ and $J_{\infty}$. If the following are met:\\ (1) $\mathcal{J}_0\subset\mathcal{G}_{0} $;\\(2)\ $\exists \eta:J_{\infty}\rightarrow G_{\infty}$, such that $\eta:J_n\setminus J_{n-1}\rightarrow G_n\setminus G_{n-1},\forall n\ge 1$ are bijections;\\(3)\ $\forall R\in J_n\setminus J_{n-1}$, $R=\pm A^k\eta(R)+\left(\in \mathcal{G}_{n-1}\right)$, $k\in \mathbb{Z}$.\\Then $\mathcal{G}_{\infty}/\mathcal{J}_{\infty}=\mathcal{G}_0/\mathcal{J}_{0}$.
    \begin{proof}
         Because (1)(2)(3), we have $J_n\subset \mathcal{G}_n,\ \forall n\ge 0$. There is a natural module homomorphism sequence, $$\mathcal{G}_0/\mathcal{J}_{0}\xrightarrow{i_0}\mathcal{G}_{1}/\mathcal{J}_{1}\xrightarrow{i_1}\mathcal{G}_{2}/\mathcal{J}_{2}\xrightarrow{i_2}\cdots \xrightarrow{i_{n-1}}\mathcal{G}_{n}/\mathcal{J}_{n}\xrightarrow{i_n}\mathcal{G}_{n+1}/\mathcal{J}_{n+1}\xrightarrow{i_{n+1}}\cdots$$ and $i_n$ is an isomorphism, because of (2)(3). \\ We claim that the natural homomorphism $i:\mathcal{G}_0/\mathcal{J}_{0}\rightarrow \mathcal{G}_{\infty}/\mathcal{J}_{\infty}$ is an isomorphism. Firstly $i$ is injective, because assuming $0\ne s\in \mathcal{G}_0/\mathcal{J}_{0}$, $i(s)=0$, then $s=\sum_i f_i(A)R_i,\ R_i\in J_{\infty}$, a sufficiently large $N$ can be found such that $R_i\in J_{N}$ so $i_{N}\cdots i_3i_2(s)=0$, but which is an isomorphism,
         contradictory. Secondly, $i$ is surjective, for $s\in \mathcal{G}_{\infty}/\mathcal{J}_{\infty}$, a sufficiently large $N$ can be found such that $i_{N}\cdots i_3i_2:\mathcal{G}_0/\mathcal{J}_{0}\rightarrow \mathcal{G}_{N}/\mathcal{J}_{N}(\ni s)$ is an isomorphism, so surjectivity holds.
         \label{Prop Reduce}
    \end{proof}
\end{lemma}

\begin{definition}
\begin{align}
    &G^{k_1,k_2}=\left\{s_1^{n_1}s_2^{n_2}s_3^{n_3}|\frac{k_1}{2}\leq n_1\leq k_1,n_2\leq\frac{k_2}{2},n_i\ge 0\right\}\nonumber\\ &\ \ \ \ \ \ \ \cup\left\{s_1^{n_1}s_2^{n_2}s_3^{n_3}|n_1< \frac{k_1}{2},n_2< \frac{k_2}{2},n_i\ge 0\right\},\ k_2\ge k_1\ge 1,\nonumber
\end{align}

$\mathcal{G}^{k_1,k_2}$ denotes the free module generated by $G^{k_1,k_2}$.
\end{definition}

    \begin{figure}
        \centering
        
    \begin{tikzpicture}
  \begin{axis}[
    xlabel={$n_1$},
    ylabel={$n_2$},
    axis lines=middle,
    xmin=0, xmax=15,
    ymin=0, ymax=15,
    xtick={0,3,6,10},
    ytick={0,5},
    xticklabels={$0$, $\frac{k_1}{2}$, $k_1$, },
    yticklabels={$0$, $\frac{k_2}{2}$},
    enlargelimits,
    legend style={at={(1.05,1)},anchor=north west}
  ]

  % Region 2: {k1/2 < n1 <= k1, n2 > k2/2}
  \addplot [
    fill=red!30,
    draw=none
] coordinates {
    (3,5) (6,5) (6,0) (3,0)
} \closedcycle;

  % Region 3: {0 <= n1 <= k1/2, n2 >= k2/2}
\addplot [
    fill=green!30,
    draw=none
] coordinates {
     (0,0) (0,3) (0,5) (3,5)
} \closedcycle;

  % Reference lines
 % n1 = k1
    \addplot[red,thick] coordinates {(3,5)(6,5)}; 

  \addplot[green, dashed] coordinates {(0,5)(3,5)};

  \end{axis}
\end{tikzpicture}
 \caption{$G^{k_1,k_2}$}
        \label{G}
    \end{figure}

\begin{theorem}
\label{thm_D^2}
    $S_{2,\infty}(D^2(k_1,k_2))\cong \bigoplus _{\alpha\in G^{k_1,k_2}}\mathbb{Z}[A^{\pm 1}]\alpha$, $k_2\ge k_1\ge 1$.
    \end{theorem}
    
    The main content of the proof is the following two lemmas. \\
\begin{definition}For convenience, we define here a morphism, \[
\begin{array}{c}
\eta : \left\{ R_{12}^{n_1, n_2, n_3} |n_i \geq 0\right\}\to \left\{ s_1^{n_1} s_2^{n_2} s_3^{n_3} |n_i \geq 0\right\} \\ \\
\phantom{\eta : } R_{12}^{n_1, n_2, n_3} \mapsto s_1^{n_1} s_2^{n_2} s_3^{n_3}
\end{array}
\]

\end{definition}
    \begin{definition}
          $J_{0,n}^{12}=J_{[0,k_1],[0,n]}^{12}\cap J^{12}_{min}$, $G^{k_1,k_2}_{0,n}=G^{k_1,k_2}\cup\eta J_{0,n}^{12}$. $\mathcal{J}_{0,n}^{12}$ are submodules of $S_{2,\infty}(H_2)$ generated by $J_{0,n}^{12}$,  $\mathcal{G}_{0,n}^{k_1,k_2}$ are free modules generated by $G_{0,n}^{k_1,k_2}$.
    \end{definition}

$S_{2,\infty}(D^2(k_1,k_2))$ is isomorphic to the free module generated by $ G^{k_1,k_2}$, where $ G^{k_1,k_2}$ is defined as 
    $$G^{k_1,k_2}=\left\{s_1^{n_1}s_2^{n_2}s_3^{n_3}|\frac{k_1}{2}\le n_1\leq k_1,n_2\le\frac{k_2}{2},n_i\ge 0\right\}\cup\left\{s_1^{n_1}s_2^{n_2}s_3^{n_3}|n_1< \frac{k_1}{2},n_2<\frac{k_2}{2},n_i\ge 0\right\},k_2\ge k_1\ge 1.$$
Again, we delay the proof of Theorem \ref{thm_D^2} after the following lemmas.

\begin{lemma}
\label{L Dkk>0 0}
$\mathcal{G}^{k_1,k_2}_{0,\infty}/\mathcal{J}_{0,\infty}^{12}=\mathcal{G}^{k_1,k_2}$.
    \begin{proof}
    \begin{enumerate}
        \item $J_{0,0}^{12}=\emptyset$, therefore $\mathcal{J}_{0,0}^{12}\subset\mathcal{G}_{0,0}$;
        \item $\eta:J_{0,n}^{12}\setminus J_{0,n-1}^{12}\rightarrow G_{0,n}^{k_1,k_2}\setminus G_{0,n-1}^{k_1,k_2},\ \forall n\ge 1$ are bijections;
        \item $\forall R_{12}^{n_1,n_2,n_3}\in J_{0,n}^{12}\setminus J_{0,n-1}^{12}=J^{12}_{[0,k_1],\left\{n\right\}}\setminus G^{k_1,k_2}$, 
    \end{enumerate}
          
        \begin{align}
        R_{12}^{n_1,n_2,n_3}=&-A^{-n_1-n_2-2}s_1^{n_1}s_2^{n_2}s_3^{n_3}+\left(-A^{-n_1-n_2+2}s_1^{n_1-2}s_2^{n_2-2}s_3^{n_3}\right.\nonumber\\&-A^{-n_1-n_2}s_1^{n_1-1}s_2^{n_2-1}s_3^{n_3+1}-A^{-n_1-n_2}s_1^{n_1-1}s_2^{n_2-1}s_3^{n_3-1}\nonumber\\&+A^{n_1+n_2-k_1-k_2-2}s_1^{-n_1+k_1}s_2^{-n_2+k_2}s_3^{n_3}+A^{n_1+n_2-k_1-k_2+2}s_1^{-n_1+k_1-2}s_2^{-n_2+k_2-2}s_3^{n_3}\nonumber\\&\left.+A^{n_1+n_2-k_1-k_2}s_1^{-n_1+k_1-1}s_2^{-n_2+k_2-1}s_3^{n_3+1}+A^{n_1+n_2-k_1-k_2}s_1^{-n_1+k_1-1}s_2^{-n_2+k_2-1}s_3^{n_3-1}\right)\nonumber\\=&-A^{-n_1-n_2-2}\eta \left(R_{12}^{n_1,n_2,n_3}\right)+\left(\in \mathcal{G}_{0,n-1}^{k_1,k_2}\right).\nonumber
        \end{align}
    Intuitively, we want to use the relations in the green and red part in Figure\ref{J} inductively up to down, showing that the rest of the elements are in Figure\ref{G}.\\
    Take $n_2=n$, as $n\ge k_2/2$,  we only need to show $s_1^{n_1-1}s_2^{n-1}s_3^{n_3-1}$ and $s_1^{-n_1+k_1-2}s_2^{-n+k_2-2}s_3^{n_3}$ are in $\mathcal{G}^{k_1,k_2}_{0,n}$ because they bound the exponents of terms except for these two terms. To show the terms are in $\mathcal{G}^{k_1,k_2}_{0,n}$ is to show  the exponent of the $n_2$-component are bounded by $0$ and $n-1$ and the exponent of the $n_1$-component are bounded by $0$ and $k_1$ according to the definition of $\mathcal{G}^{k_1,k_2}_{0,n}$. For the $n_2$-component $s_2^{n_2}$, $n_2-1$ is the highest exponent of the remaining terms, $-n_2+k_2-2$ is the lowest exponent of the remaining terms, it is clear that $n-1\le n$, so we only need to ensure $-n_2+k_2-2\ge0$, if not, $s_2^{-n_2+k_2-2}=-s_2^{n_2-k_1}$, where $n_2-k_2\ge -1$, it will infer that $-n_2+k_2\ge0$  as $s_2^{-1}=0$. For the $n_1$-component, We have that all exponents of $n_1$ of all terms are in $[-2,k_1]$ by direct checked, when $n_1=-2$ or $n_1=-1$, $s_1^{-2}=-s_1^0$, $s_1^{-1}=0$ which indicates the bound $-2$ and $k_1$ will be turned to $0$ and $k_1$ for the $n_1$-components. Then we know that $s_1^{n_1-1}s_2^{n-1}s_3^{n_3}$ and $s_1^{-n_1+k_1-2}s_2^{-n+k_2-2}s_3^{n_3}$ are actually in $\mathcal{G}^{k_1,k_2}_{0,n}$.  
     According to Proposition \ref{P base chain}, $\mathcal{G}^{k_1,k_2}_{0,\infty}/\mathcal{J}_{0,\infty}^{12}=\mathcal{G}^{k_1,k_2}_{0,0}/\mathcal{J}_{0,0}^{12}$, where $\mathcal{G}^{k_1,k_2}_{0,0}/\mathcal{J}_{0,0}^{12}=\mathcal{G}^{k_1,k_2}$.

     \end{proof}

\end{lemma}
\begin{definition}
     $J_{n}^{12}= J_{\left[0,k_1\right]\cup\left[0,n\right],[0)}^{12}\cap J^{12}_{min}$, $G^{k_1,k_2}_n=G^{k_1,k_2}\cup \eta J_n^{12}$. $\mathcal{J}_{n}^{12}$ are submodules of $S_{2,\infty}(H_2)$ generated by $J_{n}^{12}$, $\mathcal{G}_{n}^{k_1,k_2}$ are free modules generated by $G_{n}^{k_1,k_2}$.
\end{definition}

\begin{lemma}
\label{L Dkk>0 1}
$\mathcal{G}^{k_1,k_2}_{\infty}/\mathcal{J}_{\infty}^{12}=\mathcal{G}_0^{k_1,k_2}/\mathcal{J}_{0}^{12}$.
    \begin{proof}
        (1) $J_{0}^{12}=J_{0,\infty}^{12}$, therefore $\mathcal{J}_{0}^{12}\subset\mathcal{G}_{0}$; (2) $\eta:J_{n}^{12}\setminus J_{n-1}^{12}\rightarrow G_{n}^{k_1,k_2}\setminus G_{n-1}^{k_1,k_2},\ \forall n\ge 1$ are bijections; (3) $\forall R_{12}^{n_1,n_2,n_3}\in J_{n}^{12}\setminus J_{n-1}^{12}=J^{12}_{(k_1,)\cap \left\{ n \right\},[0)}$, 
        \begin{align}
        R_{12}^{n_1,n_2,n_3}=&-A^{-n_1-n_2-2}\eta \left(R_{12}^{n_1,n_2,n_3}\right)+\left(\in \mathcal{G}_{n-1}^{k_1,k_2}\right),\nonumber
        \end{align}
        where $G_{n-1}^{k_1,k_2}=\eta J^{12}_{\left[0,k_1\right]\cup\left[0,n-1\right],[0)}$. According to Proposition \ref{P base chain}, $\mathcal{G}^{k_1,k_2}_{\infty}/\mathcal{J}_{\infty}^{12}=\mathcal{G}_0^{k_1,k_2}/\mathcal{J}_{0}^{12}$.
    \end{proof}
\end{lemma}

\begin{enumerate}
    \item $J_{0}^{12}=J_{0,\infty}^{12}$, therefore $\mathcal{J}_{0}^{12}\subset\mathcal{G}_{0}$ follows form Lemma 4.4;( or we can say like this:$\mathcal{J}_{0}^{12}\subset\mathcal{G}_{0}$,the discussion are similar to those in (3), so we omit the proof here);
    \item $\eta:J_{n}^{12}\setminus J_{n-1}^{12}\rightarrow G_{n}^{k_1,k_2}\setminus G_{n-1}^{k_1,k_2},\ \forall n\ge 1$ are bijections by definitions;
    \item When $n\le k_1$ , $J^{12}_{n}\setminus J^{12}_{n-1}=\emptyset$; when $k_1\le n-1$, $J^{12}_{n}\setminus J^{12}_{n-1}=J^{12}_{\{n\},[0,)}$, in this case, $n_1=n\ge k_1+1\ge2$, $G^{k_1,k_2}_{n-1}=\eta J^{12}_{[0,n-1],[0,)}$.  $\forall R_{12}^{n_1,n_2,n_3}\in J_{n}^{12}\setminus J_{n-1}^{12}=J^{12}_{\{n\},[0,)}$, we want to show that,
\end{enumerate}
             \begin{align}
        R_{12}^{n_1,n_2,n_3}=&-A^{-n_1-n_2-2}\eta \left(R_{12}^{n_1,n_2,n_3}\right)+\left(\in \mathcal{G}_{n-1}^{k_1,k_2}\right),\nonumber
        \end{align}
       Intuitively, we want to use the relations in the blue part in Figure\ref{J} inductively from left to right, showing that the rest of the elements are in $\mathcal{G}_0^{k_1,k_2}/\mathcal{J}_{0}^{12}$\\
      Showing that the terms except for $s_1^{n_1}s_2^{n_2}s_3^{n_3}$ are in $\mathcal{G}_{n-1}^{k_1,k_2}$ is showing that the exponent of $n_1$-components is bounded by $0$ and $n-1$, the exponent of $n_2$-components are greater than $0$. For the $n_1$-component, since $n_1>k_1$, $n_1-1$ is the greatest exponent and $-n_1+k_1-2$ is the least one, so we only need to show $n_1-1,-n_1+k_1-2\in[0,n-1]$.$n_1-1$ is easy to check.  $-n_1+k_1-2\le-3$, and we have $s_1^{-n_1+k_1-2}=-s_1^{n_1-k_1}$, where $n_1-k_1\in[-1,n-1]$, which will indicates $n_1-k_1\in[0,n-1]$ and $n_1-k_1-2\in[0,n-1]$  for the same discussion above. Then we can say all the exponents of the remaining term's $n_1$-component is in $[0,n-1]$. For the $n_2$-component, we only need to care about $-n_2+k_2-2$ as it is the smallest one. If $-n_2+k_2-2<0$,then $s_2^{-n_2+k_2-2}=-s_2^{n_2-k_2}$, where $n_2-k_2>-2$, if $n_2-k_2=-1$, $s_2^{-1}=0$. Then we can know that $s_1^{-n_1+k_2-2}s_2^{-n_2+k+2-2}s_3^{n_3} \in \mathcal{G}_{n-1}^{k_1,k_2}$ and $s_2^{-1}=0$. Then we can see that the discussions of $s_1^{n_1-1}s_2^{n_2-1}s_3^{n_3-1} \in \mathcal{G}_{n-1}^{k_1,k_2}$ other terms are similar.

 \begin{proof}[\textbf{ The proof of Theorem \ref{thm_D^2}}]
     Therefore,
    \begin{align}
        S_{2,\infty}(D^2(k_1,k_2))&=S_{2,\infty}(H_2)/\mathcal{J}^{12}_{min}\nonumber\\&=\mathcal{G}^{k_1,k_2}_{\infty}/\mathcal{J}_{\infty}^{12}\xlongequal{Lemma\ \ref{L Dkk>0 1} }\mathcal{G}_0^{k_1,k_2}/\mathcal{J}_{0}^{12}\nonumber\\&=\mathcal{G}^{k_1,k_2}_{0,\infty}/\mathcal{J}_{0,\infty}^{12}\xlongequal{Lemma\ \ref{L Dkk>0 0} }\mathcal{G}^{k_1,k_2}.\nonumber
    \end{align}
 \end{proof}

\begin{corollary}

The empty link is not trivial in $S_{2,\infty}(D^2(k_1,k_2))$.
    
\end{corollary}
\begin{proof}

Clearly, $s_1^{0}s_2^{0}s_3^{0}$ is in $G^{k_1,k_2}$.
    
\end{proof}

\section{$S_{2,\infty}(S^2(k_1,k_2,k_3)),k_i\ge 2$ are finite generated}

Notice that, by comparing fundamental groups, we learn that $S^2(k_1,k_2,k_3)$ degenerates into a lens space or connected sum of lens spaces when $\min\left\{|k_1|,|k_2|,|k_3|\right\}\leq 1$. We have the following proposition.
\begin{proposition}[\cite{2016An}]
$$\begin{aligned}
    &S^2(k_1,k_2,0)=L(k_1,1)\# L(k_2,1),\\
    &S^2(k_1,k_2,1)=L(k_1k_2+k_1+k_2,k_1+1).
\end{aligned}$$
\end{proposition}

In this section, we focus on the situation when $k_1,k_2,k_3\geq 2$.

\begin{definition}
    Let $\mathcal{G}^{k_1,k_2,k_3}_{n_1,n_2,n_3}$ be the submodule of $S_{2,\infty}(S^2(k_1,k_2,k_3))$ generated by $\{s_1^{l_1}s_2^{l_2}s_3^{l_3}\}_{0\leq l_i\le n_i}$. $O(s_1^{k_1}s_2^{k_2}s_3^{k_3})$ represents some elements in $\mathcal{G}^{k_1,k_2,k_3}_{n_1,n_2,n_3}$.
\end{definition}

\begin{remark}
    $S_{2,\infty}(S^2(k_1,k_2,k_3))=\mathcal{G}^{k_1,k_2,k_3}_{\infty,\infty,\infty}$.
\end{remark}

\begin{lemma}
    $$\mathcal{G}^{k_1,k_2,k_3}_{k_1,k_2,\infty}=\mathcal{G}^{k_1,k_2,k_3}_{k_1,k_2,k_3},\quad k_i\ge 2, n_3\ge k_3.$$
    \begin{proof}
    $\forall s_1^{n_1}s_2^{n_2}s_3^{n_3+1}\in \mathcal{G}^{k_1,k_2,k_3}_{k_1,k_2,n_3+1}$, $k_1\ge n_1\ge 0$, $k_2\ge n_2\ge 0$, $n_3\ge k_3$,
        \begin{align}
            s_1^{n_1}s_2^{n_2}s_3^{n_3+1}\xlongequal{R_{23}^{n_1,n_2,n_3+1}}&-A^{4}s_1^{n_1}s_2^{n_2-2}s_3^{n_3-1}-A^2s_1^{n_1+1}s_2^{n_2-1}s_3^{n_3}-A^2s_1^{n_1-1}s_2^{n_2-1}s_3^{n_3}\nonumber\\&-A^{2n_2+2n_3-k_2-k_3+2}s_1^{n_1}s_2^{-n_2+k_2}s_3^{n_3-k_3-1}\nonumber\\&-A^{2n_2+2n_3-k_2-k_3+6}s_1^{n_1}s_2^{-n_2+k_2-2}s_3^{n_3-k_3+1}\nonumber\\&-A^{2n_2+2n_3-k_2-k_3+4}s_1^{n_1+1}s_2^{-n_2+k_2-1}s_3^{n_3-k_3}\nonumber\\&-A^{2n_2+2n_3-k_2-k_3+4}s_1^{n_1-1}s_2^{-n_2+k_2-1}s_3^{n_3-k_3}\nonumber\\=&-A^2s_1^{n_1+1}s_2^{n_2-1}s_3^{n_3}-A^{2n_2+2n_3-k_2-k_3+4}s_1^{n_1+1}s_2^{-n_2+k_2-1}s_3^{n_3-k_3}+O(s_1^{k_1}s_2^{k_2}s_3^{n_3}).\nonumber
        \end{align}
    For $s_1^{n_1+1}s_2^{n_2-1}s_3^{n_3}$,
        \begin{align}
            s_1^{n_1+1}s_2^{n_2-1}s_3^{n_3}\xlongequal{R_{13}^{n_1+1,n_2-1,n_3}}&-A^4s_1^{n_1-1}s_2^{n_2-1}s_3^{n_3-2}-A^2s_1^{n_1}s_2^{n_2}s_3^{n_3-1}-A^2s_1^{n_1}s_2^{n_2-2}s_3^{n_3-1}\nonumber\\&-A^{2n_1+2n_3-k_1-k_3+2}s_1^{-n_1+k_1-1}s_2^{n_2-1}s_3^{n_3-k_3-2}\nonumber\\&-A^{2n_1+2n_3-k_1-k_3+6}s_1^{-n_1+k_1-3}s_2^{n_2-1}s_3^{n_3-k_3}\nonumber\\&-A^{2n_1+2n_3-k_1-k_3+4}s_1^{-n_1+k_1-2}s_2^{n_2}s_3^{n_3-k_3-1}\nonumber\\&-A^{2n_1+2n_3-k_1-k_3+4}s_1^{-n_1+k_1-2}s_2^{n_2-2}s_3^{n_3-k_3-1}\nonumber\\=&O(s_1^{k_1}s_2^{k_2}s_3^{n_3-1}).\nonumber
        \end{align}
    For $s_1^{n_1+1}s_2^{-n_2+k_2-1}s_3^{n_3-k_3}$,
        \begin{align}
            s_1^{n_1+1}s_2^{-n_2+k_2-1}s_3^{n_3-k_3}\xlongequal{R_{13}^{n_1+1,-n_2+k_2-1,n_3-k_3}}&-A^4s_1^{n_1-1}s_2^{-n_2+k_2-1}s_3^{n_3-k_3-2}\nonumber\\&-A^2s_1^{n_1}s_2^{-n_2+k_2}s_3^{n_3-k_3-1}\nonumber\\&-A^2s_1^{n_1}s_2^{-n_2+k_2-2}s_3^{n_3-k_3-1}\nonumber\\&-A^{2n_1+2n_3-k_1-3k_3+2}s_1^{-n_1+k_1-1}s_2^{-n_2+k_2-1}s_3^{n_3-2k_3-2}\nonumber\\&-A^{2n_1+2n_3-k_1-3k_3+6}s_1^{-n_1+k_1-3}s_2^{-n_2+k_2-1}s_3^{n_3-2k_3}\nonumber\\&-A^{2n_1+2n_3-k_1-3k_3+4}s_1^{-n_1+k_1-2}s_2^{-n_2+k_2}s_3^{n_3-2k_3-1}\nonumber\\&-A^{2n_1+2n_3-k_1-3k_3+4}s_1^{-n_1+k_1-2}s_2^{-n_2+k_2-2}s_3^{n_3-2k_3-1}\nonumber\\=&O(s_1^{k_1}s_2^{k_2}s_3^{n_3-1}).\nonumber
        \end{align}
        Therefore, $s_1^{n_1}s_2^{n_2}s_3^{n_3+1}=O(s_1^{k_1}s_2^{k_2}s_3^{n_3})\in \mathcal{G}^{k_1,k_2,k_3}_{k_1,k_2,n_3}$. Repeat these proof steps, we have $s_1^{n_1}s_2^{n_2}s_3^{n_3}\in \mathcal{G}^{k_1,k_2,k_3}_{k_1,k_2,n_3-1}\subset \mathcal{G}^{k_1,k_2,k_3}_{k_1,k_2,n_3-2}\subset\cdots\subset\mathcal{G}^{k_1,k_2,k_3}_{k_1,k_2,k_3},\ (k_1\ge n_1\ge 0,\ k_2\ge n_2\ge 0\ ,n_3>k_3)$.
    \end{proof}
    \label{Lemma 4.5}
\end{lemma}

\begin{theorem}
    $S_{2,\infty}(S^2(k_1,k_2,k_3))=\mathcal{G}^{k_1,k_2,k_3}_{k_1,k_2,k_3}$, $k_i\ge 2$.
    \begin{proof}
        $$S_{2,\infty}(S^2(k_1,k_2,k_3))=\mathcal{G}^{k_1,k_2,k_3}_{\infty,\infty,\infty}\xlongequal{Theorem\ \ref{thm_D^2}}\mathcal{G}^{k_1,k_2,k_3}_{k_1,k_2,\infty}\xlongequal{Lemma\ \ref{Lemma 4.5}}\mathcal{G}^{k_1,k_2}_{k_1,k_2,k_3}.$$
    \end{proof}
\end{theorem}

\begin{corollary}
    $S_{2,\infty}(S^2(k_1,k_2,k_3)),k_i\ge 2$ is finite generated, minimal number of generators is less than $(k_1+1)(k_2+1)(k_3+1)$.
\end{corollary}

Next, we show that the empty link is not trivial in $S_{2,\infty}(S^{2}(k_1,k_2,k_3))$.

\subsection{The empty link is not zero}
Let $A$ be a primitive $4r^{th}$ root, or $2r^{th}$ root when $r$ is odd, of unity in $\mathbb{C}$, with $r\ge 3$. $\Delta_n=\frac{(-1)^n(A^{2(n+1)}-A^{-2(n+1)})}{A^2-A^{-2}}$, $\omega=\sum_{n=0}^{r-2}\Delta_nS_n(\alpha)$. We have 

\begin{proposition}
    Where $\mu_n=(-1)^nA^{n^2+2n}$.
    \label{Proposition 2.7}
\end{proposition}

\begin{figure}[H]
        \centering
        \includegraphics[width=0.3\linewidth]{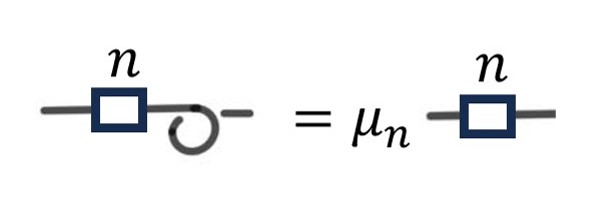}
\end{figure}

\begin{proposition}
    Where $f_n^a=(-1)^a\frac{A^{2(n+1)(a+1)}-A^{-2(n+1)(a+1)}}{A^{2(n+1)}-A^{-2(n+1)}}$.
    \label{Proposition 2.8}
\end{proposition}
\begin{figure}[H]
        \centering
        \includegraphics[width=0.3\linewidth]{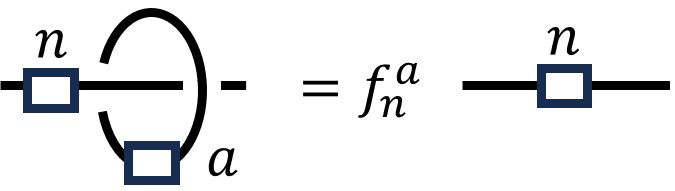}
\end{figure}
For details, we refer to \cite{0An}.
We define the following homomorphism.
\begin{definition}
    Let $\eta_A:S_{2,\infty}(S^2(k_1,k_2,k_3))\to \mathbb{C}$ be a homomorphism with each link component (including surgery link) assigned a $\omega$,  where $A=e^{\frac{\pi i}{3}}$, $r=3$, $r$ is odd (for example, the image of $a_1a_2a_3$ are as follow). According to \cite{0An}, this is a homomorphism.
     \begin{figure}[H]
        \centering
        \includegraphics[width=0.5\linewidth]{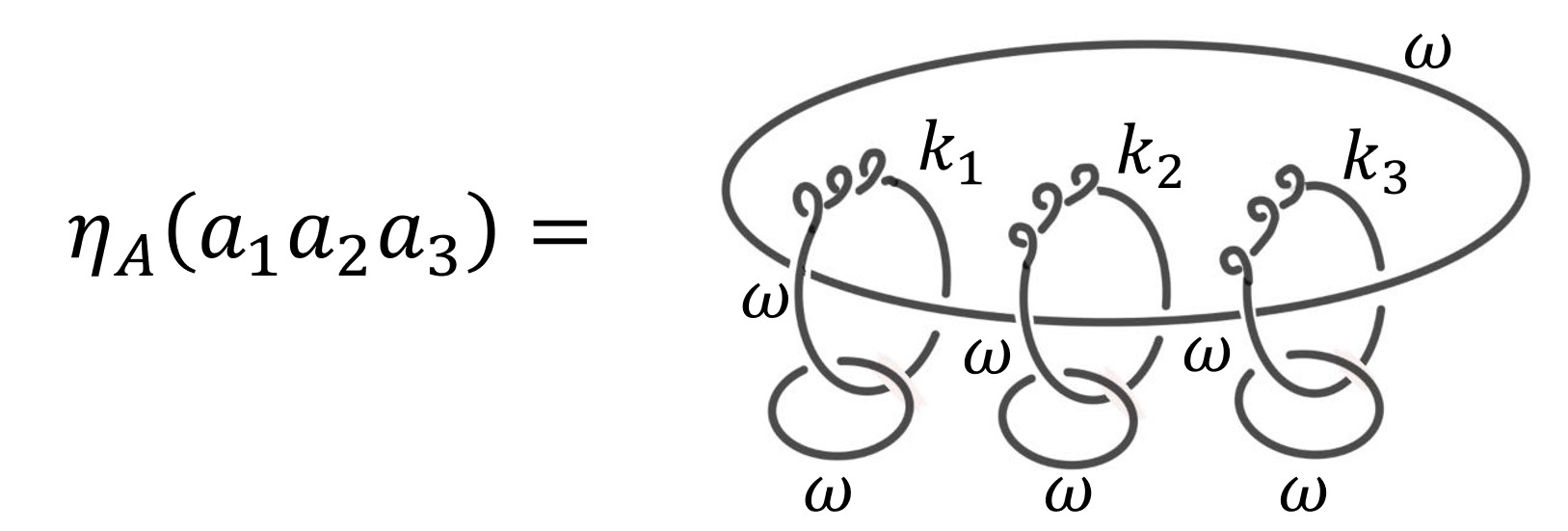}
        \caption{Example of $\eta_A(a_1a_2a_3)$}
        \label{_4_7_Def_of_eta}
    \end{figure}
\end{definition}

\begin{proposition}
    $\phi\ne 0$ in $S_{2,\infty}(S^2(k_1,k_2,k_3))$.
    \begin{proof}
        $\eta_A(\phi)=\sum_{i_0,i_1,i_2,i_3=0}^{1}\Delta_{i_0}^2\Delta_{i_1}\Delta_{i_2}\Delta_{i_3}\mu_{i_1}^{-k_1}\mu_{i_2}^{-k_2}\mu_{i_3}^{-k_3}f_{i_0}^{i_1}f_{i_0}^{i_2}f_{i_0}^{i_3}$, where $f_{0}^{0}=f_{1}^{0}=f_{0}^{1}=f_{1}^{1}=1$, $\mu_{0}=\mu_{1}=1$, $\Delta_{0}=1$, $\Delta_{1}=-1$. Therefore, $\eta_A(\phi)=\sum_{i_0,i_1,i_2,i_3=0}^{1}1=16\ne 0$.
    \end{proof}
\end{proposition}

\subsection*{Acknowledgements}

The authors would like to thank Yanqing Zou for helpful discussions. The second author was partially supported by the National Natural Science Foundation of China (Grant NO. 12131009, 12471065). The third author was partially supported by the National Natural Science Foundation of China (Grant NO. 11901229, 12371029, 22341304 and W2412041).

\section{Appendix}

\subsection{Proof of Lemma \ref{Lemma formula 0}}
\begin{proof}
        \begin{align}
            &R_{12}(-1,k_2+1,n_3)+A^2R_{12}(0,k_2,n_3-1)+A^2R_{12}(0,k_2,n_3+1)+A^4R_{12}(1,k_2-1,n_3)\nonumber\\=&-A^{-k_2-2}s_1^{-1}s_2^{k_2+1}s_3^{n_3}-A^{-k_2+2}s_1^{-3}s_2^{k_2-1}s_3^{n_3}-A^{-k_2}s_1^{-2}s_2^{k_2}s_3^{n_3+1}-A^{-k_2}s_1^{-2}s_2^{k_2}s_3^{n_3-1}\nonumber\\&-A^{-k_2}s_1^{0}s_2^{k_2}s_3^{n_3-1}-A^{-k_2+4}s_1^{-2}s_2^{k_2-2}s_3^{n_3-1}-A^{-k_2+2}s_1^{-1}s_2^{k_2-1}s_3^{n_3}-A^{-k_2+2}s_1^{-1}s_2^{k_2-1}s_3^{n_3-2}\nonumber\\&-A^{-k_2}s_1^{0}s_2^{k_2}s_3^{n_3+1}-A^{-k_2+4}s_1^{-2}s_2^{k_2-2}s_3^{n_3+1}-A^{-k_2+2}s_1^{-1}s_2^{k_2-1}s_3^{n_3}-A^{-k_2+2}s_1^{-1}s_2^{k_2-1}s_3^{n_3}\nonumber\\&-A^{-k_2+2}s_1^{1}s_2^{k_2-1}s_3^{n_3}-A^{-k_2+6}s_1^{-1}s_2^{k_2-3}s_3^{n_3}-A^{-k_2+4}s_1^{0}s_2^{k_2-2}s_3^{n_3+1}-A^{-k_2+4}s_1^{0}s_2^{k_2-2}s_3^{n_3-1}\nonumber\\=&0.\nonumber
        \end{align}
        Similarly,
        \begin{align}
            &R_{12}(k_1+1,-1,n_3)+A^2R_{12}(k_1,0,n_3-1)+A^2R_{12}(k_1,0,n_3+1)+A^4R_{12}(k_1-1,1,n_3)=0.\nonumber
        \end{align}
        Therefore,
        \begin{align}
        &R_{12}^{-1,k_2+1,n_3}+A^{2}R_{12}^{0,k_2,n_3-1}+A^2R_{12}^{0,k_2,n_3+1}+A^{4}R_{12}^{1,k_2-1,n_3}\nonumber\\=&R_{12}(-1,k_2+1,n_3)+A^2R_{12}(0,k_2,n_3-1)+A^2R_{12}(0,k_2,n_3+1)+A^4R_{12}(1,k_2-1,n_3)\nonumber\\&-
        \left(R_{12}(k_1+1,-1,n_3)+A^2R_{12}(k_1,0,n_3-1)+A^2R_{12}(k_1,0,n_3+1)+A^4R_{12}(k_1-1,1,n_3)\right)\nonumber\\=&0.\nonumber
        \end{align}
    \end{proof}
\subsection{Proof of Lemma \ref{Lemma formula 01}}
\begin{proof}

\begin{definition}
\begin{align}
    F_{u}^{n_1,n_2,n_3}=-A^{-n_1-n_2-2}\sum_{j=0}^{u}s_1^{n_1}s_2^{n_2}s_3^{n_3+2j}-A^{-n_1-n_2}\sum_{j=0}^{u-1}s_1^{n_1-1}s_2^{n_2-1}s_3^{n_3+1+2j},\nonumber\\\tilde{F}_{u}^{n_1,n_2,n_3}=-A^{-n_1-n_2-2}\sum_{j=0}^{u}s_1^{n_1}s_2^{n_2}s_3^{n_3+2j}-A^{-n_1-n_2}\sum_{j=0}^{u+1}s_1^{n_1-1}s_2^{n_2-1}s_3^{n_3-1+2j}.\nonumber
\end{align}

\end{definition}

\begin{lemma}
\label{Lemma F+F}
    \begin{align}
    &F_{n_1+1}^{-1,n_2,n_3}+A^2F_{n_1}^{0,n_2-1,n_3+1}=0,\nonumber\\
    &\tilde{F}_{n_1-1}^{-1,n_2,n_3}+A^2\tilde{F}_{n_1}^{0,n_2-1,n_3-1}=0.\nonumber
    \end{align}
    \begin{proof}
        Direct calculation.
    \end{proof}
\end{lemma}

\begin{lemma}
\label{Lemma 4+F}
    \begin{align}
        A^{2n_1}\tilde{F}_{n_1}^{k,n_1-n_2-1,n_3}+A^{2n_1+2}\tilde{F}_{n_1-1}^{k-1,n_1-n_2,n_3+1}+A^{2n_2+2}F_{n_1+1}^{k-1,n_2-n_1,n_3-1}+A^{2n_2}F_{n_1}^{k,n_2-n_1-1,n_3}=0.\nonumber
    \end{align}
    \begin{proof}
        \begin{align}
            Left=&-A^{-k+n_1+n_2-1}\sum_{j=0}^{n_1}s_1^{k}s_2^{n_1-n_2-1}s_3^{n_3+2j}-A^{-k+n_1+n_2+1}\sum_{j=0}^{n_1+1}s_1^{k-1}s_2^{n_1-n_2-2}s_3^{n_3-1+2j}\nonumber\\&-A^{-k+n_1+n_2+1}\sum_{j=0}^{n_1-1}s_1^{k-1}s_2^{n_1-n_2}s_3^{n_3+1+2j}-A^{-k+n_1+n_2+3}\sum_{j=0}^{n_1}s_1^{k-2}s_2^{n_1-n_2-1}s_3^{n_3+2j}\nonumber\\&-A^{-k+n_1+n_2+1}\sum_{j=0}^{n_1+1}s_1^{k-1}s_2^{n_2-n_1}s_3^{n_3-1+2j}-A^{-k+n_1+n_2+3}\sum_{j=0}^{n_1}s_1^{k-2}s_2^{n_2-n_1-1}s_3^{n_3+2j}\nonumber\\&-A^{-k+n_1+n_2-1}\sum_{j=0}^{n_1}s_1^{k}s_2^{n_2-n_1-1}s_3^{n_3+2j}-A^{-k+n_1+n_2+1}\sum_{j=0}^{n_1-1}s_1^{k-1}s_2^{n_2-n_1-2}s_3^{n_3+1+2j}\nonumber\\=&0.\nonumber
        \end{align}
    \end{proof}
\end{lemma}

\begin{lemma}
\label{Lemma R=F+F=F+F}
    \begin{align}
        \sum_{j=0}^{u}R_{12}(n_1,n_2,n_3+2j)&=F_{u}^{n_1,n_2,n_3}+F_{u+1}^{n_1-1,n_2-1,n_3-1}\nonumber\\&=\tilde{F}_{u}^{n_1,n_2,n_3}+\tilde{F}_{u-1}^{n_1-1,n_2-1,n_3+1},\ u\ge 0.\nonumber
    \end{align}
    \begin{proof}
        \begin{align}
            \sum_{j=0}^{u}R_{12}(n_1,n_2,n_3+2j)=&-A^{-n_1-n_2-2}\sum_{j=0}^{u}s_1^{n_1}s_2^{n_2}s_3^{n_3+2j}-A^{-n_1-n_2+2}\sum_{j=0}^{u}s_1^{n_1-2}s_2^{n_2-2}s_3^{n_3+2j}\nonumber\\&-A^{-n_1-n_2}\sum_{j=0}^{u}s_1^{n_1-1}s_2^{n_2-1}s_3^{n_3+1+2j}-A^{-n_1-n_2}\sum_{j=0}^{u}s_1^{n_1-1}s_2^{n_2-1}s_3^{n_3-1+2j}\nonumber\\=&-A^{-n_1-n_2-2}\sum_{j=0}^{u}s_1^{n_1}s_2^{n_2}s_3^{n_3+2j}-A^{-n_1-n_2+2}\sum_{j=0}^{u}s_1^{n_1-2}s_2^{n_2-2}s_3^{n_3+2j}\nonumber\\&-A^{-n_1-n_2}\sum_{j=0}^{u-1}s_1^{n_1-1}s_2^{n_2-1}s_3^{n_3+1+2j}-A^{-n_1-n_2}\sum_{j=0}^{u+1}s_1^{n_1-1}s_2^{n_2-1}s_3^{n_3-1+2j}\nonumber \\=&\left(-A^{-n_1-n_2-2}\sum_{j=0}^{u}s_1^{n_1}s_2^{n_2}s_3^{n_3+2j}-A^{-n_1-n_2}\sum_{j=0}^{u-1}s_1^{n_1-1}s_2^{n_2-1}s_3^{n_3+1+2j}\right)\nonumber\\&+\left(-A^{-n_1-n_2}\sum_{j=0}^{u+1}s_1^{n_1-1}s_2^{n_2-1}s_3^{n_3-1+2j}-A^{-n_1-n_2+2}\sum_{j=0}^{u}s_1^{n_1-2}s_2^{n_2-2}s_3^{n_3+2j}\right)\nonumber\\=&F_{u}^{n_1,n_2,n_3}+F_{u+1}^{n_1-1,n_2-1,n_3-1}.\nonumber
        \end{align}
        \begin{align}
            \sum_{j=0}^{u}R_{12}(n_1,n_2,n_3+2j)=&\left(-A^{-n_1-n_2-2}\sum_{j=0}^{u}s_1^{n_1}s_2^{n_2}s_3^{n_3+2j}-A^{-n_1-n_2}\sum_{j=0}^{u+1}s_1^{n_1-1}s_2^{n_2-1}s_3^{n_3-1+2j}\right)\nonumber\\&+\left(-A^{-n_1-n_2}\sum_{j=0}^{u-1}s_1^{n_1-1}s_2^{n_2-1}s_3^{n_3+1+2j}-A^{-n_1-n_2+2}\sum_{j=0}^{u}s_1^{n_1-2}s_2^{n_2-2}s_3^{n_3+2j}\right)\nonumber\\=&\tilde{F}_{u}^{n_1,n_2,n_3}+\tilde{F}_{u-1}^{n_1-1,n_2-1,n_3+1}.\nonumber
        \end{align}
    \end{proof}
\end{lemma}

\begin{lemma}
\label{Lemma R12}
    \begin{align}
        &\sum_{i=u_0}^{u_1}\sum_{j=0}^{i+c}(-1)^iR_{12}(n_1-i,n_2-i,n_3-i+2j)\nonumber\\&=(-1)^{u_0}F_{u_0+c}^{n_1-u_0,n_2-u_0,n_3-u_0}+(-1)^{u_1}F_{u_1+c+1}^{n_1-u_1-1,n_2-u_1-1,n_3-u_1-1},\nonumber\\&\sum_{i=u_0}^{u_1}\sum_{j=0}^{i+c}(-1)^iR_{12}(n_1+i,n_2+i,n_3-i+2j)\nonumber\\&=(-1)^{u_0}\tilde{F}_{u_0+c-1}^{n_1+u_0-1,n_2+u_0-1,n_3-u_0+1}+(-1)^{u_1}\tilde{F}_{u_1+c}^{n_1+u_1,n_2+u_1,n_3-u_1},\nonumber
    \end{align}
    where $\ u_1\ge u_0,\ u_0+c\ge 0$.
    \begin{proof}
        \begin{align}
            &\sum_{i=u_0}^{u_1}\sum_{j=0}^{i+c}(-1)^iR_{12}(n_1-i,n_2-i,n_3-i+2j)\nonumber\\&=\sum_{i=u_0}^{u_1}(-1)^i\sum_{j=0}^{i+c}R_{12}(n_1-i,n_2-i,n_3-i+2j)\nonumber\\&\xlongequal{Lemma\ \ref{Lemma R=F+F=F+F} }\sum_{i=u_0}^{u_1}(-1)^i \left(F_{i+c}^{n_1-i,n_2-i,n_3-i}+F_{i+c+1}^{n_1-i-1,n_2-i-1,n_3-i-1}\right)\nonumber\\&=\sum_{i=u_0}^{u_1}(-1)^i F_{i+c}^{n_1-i,n_2-i,n_3-i}-\sum_{i=u_0+1}^{u_1+1}(-1)^iF_{i+c}^{n_1-i,n_2-i,n_3-i}\nonumber\\&=(-1)^{u_0}F_{u_0+c}^{n_1-u_0,n_2-u_0,n_3-u_0}+(-1)^{u_1}F_{u_1+c+1}^{n_1-u_1-1,n_2-u_1-1,n_3-u_1-1},\nonumber
        \end{align}
        \begin{align}
            &\sum_{i=u_0}^{u_1}\sum_{j=0}^{i+c}(-1)^iR_{12}(n_1+i,n_2+i,n_3-i+2j)\nonumber\\&=\sum_{i=u_0}^{u_1}(-1)^i\sum_{j=0}^{i+c}R_{12}(n_1+i,n_2+i,n_3-i+2j)\nonumber\\&\xlongequal{Lemma\ \ref{Lemma R=F+F=F+F} } \sum_{i=u_0}^{u_1}(-1)^i \left(\tilde{F}_{i+c}^{n_1+i,n_2+i,n_3-i}+\tilde{F}_{i+c-1}^{n_1+i-1,n_2+i-1,n_3-i+1}\right)\nonumber\\&=\sum_{i=u_0}^{u_1}(-1)^i \tilde{F}_{i+c}^{n_1+i,n_2+i,n_3-i}-\sum_{i=u_0-1}^{u_1-1}(-1)^i\tilde{F}_{i+c}^{n_1+i,n_2+i,n_3-i}\nonumber\\&=(-1)^{u_0}\tilde{F}_{u_0+c-1}^{n_1+u_0-1,n_2+u_0-1,n_3-u_0+1}+(-1)^{u_1}\tilde{F}_{u_1+c}^{n_1+u_1,n_2+u_1,n_3-u_1}\nonumber.
        \end{align}
    \end{proof}
\end{lemma}

    \begin{align}
        &R_{12}^{-n_1,k_2+n_2,n_3}\nonumber\\&+A^{2n_1-2}\sum_{i=0}^{n_1-2}\sum_{j=0}^{i}(-1)^iR_{12}^{n_1-2-i,k_2+n_2-i,n_3-i+2j}+A^{2n_1}\sum_{i=0}^{n_1-1}\sum_{j=0}^{i+1}(-1)^iR_{12}^{n_1-1-i,k_2+n_2-1-i,n_3-1-i+2j}\nonumber\\&+A^{2n_1}\sum_{i=1}^{n_1-2}\sum_{j=0}^{i-1}(-1)^iR_{12}^{n_1-1-i,k_2+n_2-1-i,n_3+1-i+2j}+A^{2n_1+2}\sum_{i=0}^{n_1-1}\sum_{j=0}^{i}(-1)^iR_{12}^{n_1-i,k_2+n_2-2-i,n_3-i+2j}\nonumber\\&\xlongequal{\ Lemma\ \ref{Lemma R12}}\nonumber\\&R_{12}(-n_1,k_2+n_2,n_3)-R_{12}(k_1+n_1,-n_2,n_3)\nonumber\\&+A^{2n_1-2}\left( F_{0}^{n_1-2,k_2+n_2,n_3}+(-1)^{n_1}F_{n_1-1}^{-1,k_2+n_2-n_1+1,n_3-n_1+1}\right.\nonumber\\&\left.-\tilde{F}_{-1}^{k_1-n_1+1,-n_2-1,n_3+1}-(-1)^{n_1}\tilde{F}_{n_1-2}^{k_1,-n_2+n_1-2,n_3-n_1+2}\right)\nonumber\\&+A^{2n_1}\left( F_{1}^{n_1-1,k_2+n_2-1,n_3-1}-(-1)^{n_1}F_{n_1+1}^{-1,k_2+n_2-n_1-1,n_3-n_1-1}\right.\nonumber\\&-\tilde{F}_{0}^{k_1-n_1,-n_2,n_3}+(-1)^{n_1}\tilde{F}_{n_1}^{k_1,-n_2+n_1,n_3-n_1}\nonumber\\& -F_{0}^{n_1-2,k_2+n_2-2,n_3}+(-1)^{n_1}F_{n_1-2}^{0,k_2+n_2-n_1,n_3-n_1+2}\nonumber\\&\left.+\tilde{F}_{-1}^{k_1-n_1+2,-n_2+1,n_3+1}-(-1)^{n_1}\tilde{F}_{n_1-3}^{k_1-1,-n_2+n_1-1,n_3-n_1+3}\right)\nonumber\\&+A^{2n_1+2}\left(F_{0}^{n_1,k_2+n_2-2,n_3}-(-1)^{n_1}F_{n_1}^{0,k_2+n_2-n_1-2,n_3-n_1}\right.\nonumber\\&\left.-\tilde{F}_{-1}^{k_1-n_1-1,-n_2+1,n_3+1}+(-1)^{n_1}\tilde{F}_{n_1-1}^{k_1-1,-n_2+n_1+1,n_3-n_1+1}\right),\nonumber
    \end{align}
    \begin{align}
        &-A^{2n_1+2n_2}R_{12}^{-n_1+k_1,n_2,n_3}\nonumber\\&-A^{2n_2+2}\sum_{i=0}^{n_1-2}\sum_{j=0}^{i}(-1)^iR_{12}^{n_1+k_1-2-i,n_2-i,n_3-i+2j}-A^{2n_2}\sum_{i=0}^{n_1-1}\sum_{j=0}^{i+1}(-1)^iR_{12}^{n_1+k_1-1-i,n_2-1-i,n_3-1-i+2j}\nonumber \\ &-A^{2n_2}\sum_{i=1}^{n_1-2}\sum_{j=0}^{i-1}(-1)^iR_{12}^{n_1+k_1-1-i,n_2-1-i,n_3+1-i+2j}-A^{2n_2-2}\sum_{i=0}^{n_1-1}\sum_{j=0}^{i}(-1)^iR_{12}^{n_1+k_1-i,n_2-2-i,n_3-i+2j}\nonumber\\&\xlongequal{\ Lemma\ \ref{Lemma R12}}\nonumber\\&-A^{2n_1+2n_2}R_{12}(-n_1+k_1,n_2,n_3)+A^{2n_1+2n_2}R_{12}(n_1,-n_2+k_2,n_3)\nonumber\\&+A^{2n_2+2}\left( F_{0}^{n_1+k_1-2,n_2,n_3}+(-1)^{n_1}F_{n_1-1}^{k_1-1,n_2-n_1+1,n_3-n_1+1}\right.\nonumber\\&\left.-\tilde{F}_{-1}^{-n_1+1,k_2-n_2-1,n_3+1}-(-1)^{n_1}\tilde{F}_{n_1-2}^{0,k_2-n_2+n_1-2,n_3-n_1+2}\right)\nonumber\\&+A^{2n_2}\left( F_{1}^{n_1+k_1-1,n_2-1,n_3-1}-(-1)^{n_1}F_{n_1+1}^{k_1-1,n_2-n_1-1,n_3-n_1-1}\right.\nonumber\\&-\tilde{F}_{0}^{-n_1,k_2-n_2,n_3}+(-1)^{n_1}\tilde{F}_{n_1}^{0,k_2-n_2+n_1,n_3-n_1}\nonumber\\&-F_{0}^{n_1+k_1-2,n_2-2,n_3}+(-1)^{n_1}F_{n_1-2}^{k_1,n_2-n_1,n_3-n_1+2}\nonumber\\&\left.+\tilde{F}_{-1}^{-n_1+2,k_2-n_2+1,n_3+1}-(-1)^{n_1}\tilde{F}_{n_1-3}^{-1,k_2-n_2+n_1-1,n_3-n_1+3}\right)\nonumber\\&+A^{2n_2-2}\left(F_{0}^{n_1+k_1,n_2-2,n_3}-(-1)^{n_1}F_{n_1}^{k_1,n_2-n_1-2,n_3-n_1}\right.\nonumber
        \\&\left.-\tilde{F}_{-1}^{-n_1-1,k_2-n_2+1,n_3+1}+(-1)^{n_1}\tilde{F}_{n_1-1}^{-1,k_2-n_2+n_1+1,n_3-n_1+1}\right).\nonumber
    \end{align}

    Therefore,
        \begin{align}
            Left&=\left(R_{12}(-n_1,k_2+n_2,n_3)+A^{2n_1-2}F_0^{n_1-2,k_2+n_2,n_3}+A^{2n_1}F_1^{n_1-1,k_2+n_2-1,n_3-1}\right.\nonumber\\&\left.-A^{2n_1}F_0^{n_1-2,k_2+n_2-2,n_3}+A^{2n_1+2}F_0^{n_1,k_2+n_2-2,n_3}\right)\nonumber\\&-\left(R_{12}(k_1+n_1,-n_2,n_3)+A^{2n_2+2}F_0^{n_1+k_1-2,n_2,n_3}+A^{2n_2}F_1^{n_1+k_1-1,n_2-1,n_3-1}\right.\nonumber\\&\left.-A^{2n_2}F_0^{n_1+k_1-2,n_2-2,n_3}+A^{2n_2-2}F_0^{n_1+k_1,n_2-2,n_3}\right)\nonumber\\&-A^{2n_1+2n_2}\left(R_{12}(-n_1+k_1,n_2,n_3)+A^{-2n_2-2}\tilde{F}_{-1}^{k_1-n_1+1,-n_2-1,n_3+1}+A^{-2n_2}\tilde{F}_0^{k_1-n_1,-n_2,n_3}\right.\nonumber\\&-A^{-2n_2}F_{-1}^{-n_1+1,-n_2+1,n_3+1}\left.+A^{-2n_2+2}F_{-1}^{k_1-n_1-1,-n_2+1,n_3+1}\right)\nonumber\\&+A^{2n_1+2n_2}\left(R_{12}(n_1,-n_2+k_2,n_3)+A^{-2n_1+2}\tilde{F}_{-1}^{-n_1+1,k_2-n_2-1,n_3+1}+A^{-2n_1}\tilde{F}_0^{-n_1,k_2-n_2,n_3}\right.\nonumber\\&\left.-A^{-2n_1}\tilde{F}_{-1}^{-n_1+1,k_2-n_2+1,n_3+1}+A^{-2n_1-2}\tilde{F}_{-1}^{-n_1-1,k_2-n_2+1,n_3+1}\right)\nonumber\\&-(-1)^{n_1}\left(A^{2n_1-2}\tilde{F}_{n_1-2}^{k_1,-n_2+n_1-2,n_3-n_1+2}+A^{2n_1}\tilde{F}_{n_1-3}^{k_1-1,-n_2+n_1-1,n_3-n_1+3}\right.\nonumber\\&\left.+A^{2n_2+2}F_{n_1-1}^{k_1-1,n_2-n_1+1,n_3-n_1+1}+A^{2n_2}F_{n_1-2}^{k_1,n_2-n_1,n_3-n_1+2}\right)\nonumber\\&+(-1)^{n_1}\left(A^{2n_1}\tilde{F}_{n_1}^{k_1,-n_2+n_1,n_3-n_1}+A^{2n_1+2}\tilde{F}_{n_1-1}^{k_1-1,-n_2+n_1+1,n_3-n_1+1}\right.\nonumber\\&\left.+A^{2n_2}F_{n_1+1}^{k_1-1,n_2-n_1-1,n_3-n_1-1}+A^{2n_2-2}F_{n_1}^{k_1,n_2-n_1-2,n_3-n_1}\right)\nonumber\\&+(-1)^{n_1}A^{2n_1-2}\left(F_{n_1-1}^{-1,k_2+n_2-n_1+1,n_3-n_1+1}+A^2F_{n_1-2}^{0,k_2+n_2-n_1,n_3-n_1+2}\right)\nonumber\\&-(-1)^{n_1}A^{2n_1}\left(F_{n_1+1}^{-1,k_2+n_2-n_1-1,n_3-n_1-1}+A^2F_{n_1}^{0,k_2+n_2-n_1-2,n_3-n_1}\right)\nonumber\\&+(-1)^{n_1}A^{2n_2}\left(A^2\tilde{F}_{n_1-2}^{0,k_2-n_2+n_1-2,n_3-n_1+2}+\tilde{F}_{n_1-3}^{-1,k_2-n_2+n_1-1,n_3-n_1+3}\right)\nonumber\\&-(-1)^{n_1}A^{2n_2-2}\left(A^2\tilde{F}_{n_1}^{0,k_2-n_2+n_1,n_3-n_1}+\tilde{F}_{n_1-1}^{-1,k_2-n_2+n_1+1,n_3-n_1+1}\right)\nonumber\\&\xlongequal{Lemma\ \ref{Lemma F+F}\ \ref{Lemma 4+F}}0.\nonumber
        \end{align}
    \end{proof}
    
\subsection{Proof of Lemma \ref{Lemma formula 1}}
    \begin{proof}
    Only prove the first equation of Lemma \ref{Lemma formula 1}, the method for the second equation is the same. We prove the following lemmas.
    \begin{lemma}
    \label{Lemma 1_1}
        \begin{align}
            \sum_{i=u_0}^{u_1}(-1)^iR_{12}(n_1-i,n_2-i,i)=&-(-1)^{u_0}A^{-n_1-n_2+2u_0-2}s_1^{n_1-u_0}s_2^{n_2-u_0}s_3^{u_0}\nonumber\\&-(-1)^{u_0}A^{-n_1-n_2+2u_0}s_1^{n_1-u_0-1}s_2^{n_2-u_0-1}s_3^{u_0-1}\nonumber\\&-(-1)^{u_1}A^{-n_1-n_2+2u_1}s_1^{n_1-u_1-1}s_2^{n_2-u_1-1}s_3^{u_1+1}\nonumber\\&-(-1)^{u_1}A^{-n_1-n_2+2u_1+2}s_1^{n_1-u_1-2}s_2^{n_2-u_1-2}s_3^{u_1}\nonumber,\\\sum_{i=u_0}^{u_1}(-1)^iR_{12}(n_1+i,n_2+i,i)=&-(-1)^{u_0}A^{-n_1-n_2-2u_0-2}s_1^{n_1+u_0}s_2^{n_2+u_0}s_3^{u_0}\nonumber\\&-(-1)^{u_0}A^{-n_1-n_2-2u_0}s_1^{n_1+u_0-1}s_2^{n_2+u_0-1}s_3^{u_0+1}\nonumber\\&-(-1)^{u_1}A^{-n_1-n_2-2u_1-2}s_1^{n_1+u_1}s_2^{n_2+u_1}s_3^{u_1}\nonumber\\&-(-1)^{u_1}A^{-n_1-n_2-2u_1}s_1^{n_1+u_1-1}s_2^{n_2+u_1-1}s_3^{u_1+1}\nonumber,
        \end{align}
        where $u_0\leq u_1$, and there are similar equations for $R_{13}$ and $R_{23}$.
        \begin{proof}
        We only prove the first equation.
            \begin{align}
                \sum_{i=u_0}^{u_1}(-1)^iR_{12}(n_1-i,n_2-i,i)=&\sum_{i=u_0}^{u_1}(-1)^i\left(-A^{-n_1-n_2+2i-2}s_1^{n_1-i}s_2^{n_2-i}s_3^{i}\right.\nonumber\\&-A^{-n_1-n_2+2i+2}s_1^{n_1-i-2}s_2^{n_2-i-2}s_3^{i}\nonumber\\&-A^{-n_1-n_2+2i}s_1^{n_1-i-1}s_2^{n_2-i-1}s_3^{i+1}\nonumber\\&\left.-A^{-n_1-n_2+2i}s_1^{n_1-i-1}s_2^{n_2-i-1}s_3^{i-1}\right)\nonumber\\=&\sum_{i=u_0}^{u_1}(-1)^i\left(-A^{-n_1-n_2+2i-2}s_1^{n_1-i}s_2^{n_2-i}s_3^{i}\right.\nonumber\\&\left.-A^{-n_1-n_2+2i}s_1^{n_1-i-1}s_2^{n_2-i-1}s_3^{i-1}\right)\nonumber\\-&\sum_{i=u_0+1}^{u_1+1}(-1)^i\left(-A^{-n_1-n_2+2i-2}s_1^{n_1-i}s_2^{n_2-i}s_3^{i}\right.\nonumber\\&\left.-A^{-n_1-n_2+2i}s_1^{n_1-i-1}s_2^{n_2-i-1}s_3^{i-1}\right)\nonumber\\=&-(-1)^{u_0}A^{-n_1-n_2+2u_0-2}s_1^{n_1-u_0}s_2^{n_2-u_0}s_3^{u_0}\nonumber\\&-(-1)^{u_0}A^{-n_1-n_2+2u_0}s_1^{n_1-u_0-1}s_2^{n_2-u_0-1}s_3^{u_0-1}\nonumber\\&-(-1)^{u_1}A^{-n_1-n_2+2u_1}s_1^{n_1-u_1-1}s_2^{n_2-u_1-1}s_3^{u_1+1}\nonumber\\&-(-1)^{u_1}A^{-n_1-n_2+2u_1+2}s_1^{n_1-u_1-2}s_2^{n_2-u_1-2}s_3^{u_1}\nonumber.
            \end{align}
        \end{proof}
    \end{lemma}
    \begin{lemma}
    \label{Lemma 1_2}
    \begin{align}
        &R_{23}(n_1,n_2,n_3)\nonumber\\&=A^{n_1-n_3}\sum_{i=0}^{n_1}(-1)^i\left(R_{12}(n_1-i,n_2-i,n_3+i)+A^2R_{12}(n_1+1-i,n_2-1-i,n_3-1+i)\right)\nonumber\\&=A^{n_1-n_2}\sum_{i=0}^{n_1}(-1)^i\left(R_{13}(n_1-i,n_2+i,n_3-i)+A^2R_{13}(n_1+1-i,n_2-1+i,n_3-1-i)\right),\nonumber
    \end{align}
    where $n_1\ge 0$.
    \begin{proof}
         We only prove the first equation.
         \begin{align}
             Right&=(-1)^{n_3}A^{n_1-n_3}\sum_{i=n_3}^{n_1+n_3}(-1)^iR_{12}(n_1+n_3-i,n_2+n_3-i,i)\nonumber\\&\quad-(-1)^{n_3}A^{n_1-n_3+2}\sum_{i=n_3-1}^{n_1+n_3-1}(-1)^iR_{12}(n_1+n_3-i,n_2+n_3-2-i,i)\nonumber\\ &\xlongequal{Lemma\ \ref{Lemma 1_1}}\nonumber\\&\quad(-1)^{n_3}A^{n_1-n_3}\left(-(-1)^{n_3}A^{-n_1-n_2-2}s_1^{n_1}s_2^{n_2}s_3^{n_3}-(-1)^{n_3}A^{-n_1-n_2}s_1^{n_1-1}s_2^{n_2-1}s_3^{n_3-1}\right.\nonumber\\&\quad\left.-(-1)^{n_1+n_3}A^{n_1-n_2}s_1^{-1}s_2^{n_2-n_1-1}s_3^{n_1+n_3+1}-(-1)^{n_1+n_3}A^{n_1-n_2+2}s_1^{-2}s_2^{n_2-n_1-2}s_3^{n_1+n_3}\right)\nonumber\\&\quad-(-1)^{n_3}A^{n_1-n_3+2}\left(-(-1)^{n_3-1}A^{-n_1-n_2-2}s_1^{n_1+1}s_2^{n_2-1}s_3^{n_3-1}-(-1)^{n_3-1}A^{-n_1-n_2}s_1^{n_1}s_2^{n_2-2}s_3^{n_3-2}\right.\nonumber\\&\quad\left.-(-1)^{n_1+n_3-1}A^{n_1-n_2}s_1^{0}s_2^{n_2-n_1-2}s_3^{n_1+n_3}-(-1)^{n_1+n_3-1}A^{n_1-n_2+2}s_1^{-1}s_2^{n_2-n_1-3}s_3^{n_1+n_3-2}\right)\nonumber\\&=-A^{-n_2-n_3-2}s_1^{n_1}s_2^{n_2}s_3^{n_3}-A^{-n_2-n_3+2}s_1^{n_1}s_2^{n_2-2}s_3^{n_3-2}\nonumber\\&\quad-A^{-n_2-n_3}s_1^{n_1+1}s_2^{n_2-1}s_3^{n_3-1}-A^{-n_2-n_3}s_1^{n_1-1}s_2^{n_2-1}s_3^{n_3-1}\nonumber\\&=Left.\nonumber
         \end{align}
    \end{proof}
    \end{lemma}
    \begin{lemma}
    \label{Lemma 1_3}
        \begin{align}
            &A^{-n_3}\sum_{i=0}^{n_1}(-1)^i\left(R_{12}(-n_1+k_1+i,n_2+i,n_3+i)+A^2R_{12}(-n_1+k_1-1+i,n_2+1+i,n_3-1+i)\right)\nonumber\\=&A^{-n_2}\sum_{i=0}^{n_1}(-1)^i\left(R_{13}(-n_1+k_1+i,n_2+i,n_3+i)+A^2R_{13}(-n_1+k_1-1+i,n_2-1+i,n_3+1+i)\right),\nonumber
        \end{align}
        where $n_1\ge0$.
        \begin{proof}
            \begin{align}
                Left&=(-1)^{n_3}A^{-n_3}\sum_{i=n_3}^{n_1+n_3}(-1)^iR_{12}(-n_1-n_3+k_1+i,n_2-n_3+i,i)\nonumber\\ &\quad-(-1)^{n_3}A^{-n_3+2}\sum_{i=n_3-1}^{n_1+n_3-1}(-1)^{i}R_{12}(-n_1-n_3+k_1+i,n_2-n_3+2+i,i)\nonumber\\&\xlongequal{Lemma\ \ref{Lemma 1_1}}\nonumber\\&\quad(-1)^{n_3}A^{-n_3}\left(-(-1)^{n_3}A^{n_1-n_2-k_1-2}s_1^{-n_1+k_1}s_2^{n_2}s_3^{n_3}-(-1)^{n_3}A^{n_1-n_2-k_1}s_1^{-n_1+k_1-1}s_2^{n_2-1}s_3^{n_3+1}\right.\nonumber\\&\quad\left.-(-1)^{n_1+n_3}A^{-n_1-n_2-k_1-2}s_1^{k_1}s_2^{n_1+n_2}s_3^{n_1+n_3}-(-1)^{n_1+n_3}A^{-n_1-n_2-k_1}s_1^{k_1-1}s_2^{n_1+n_2-1}s_3^{n_1+n_3+1}\right)\nonumber\\&\quad-(-1)^{n_3}A^{-n_3+2}\left(-(-1)^{n_3-1}A^{n_1-n_2-k_1-2}s_1^{-n_1+k_1-1}s_2^{n_2+1}s_3^{n_3-1}-(-1)^{n_3-1}A^{n_1-n_2-k_1}s_1^{-n_1+k_1-2}s_2^{n_2}s_3^{n_3}\right.\nonumber\\&\quad\left.-(-1)^{n_1+n_3-1}A^{-n_1-n_2-k_1-2}s_1^{k_1-1}s_2^{n_1+n_2+1}s_3^{n_1+n_3-1}-(-1)^{n_1+n_3-1}A^{-n_1-n_2-k_1}s_1^{k_1-2}s_2^{n_1+n_2}s_3^{n_1+n_3}\right)\nonumber\\&=-A^{n_1-n_2-n_3-k_1-2}s_1^{-n_1+k_1}s_2^{n_2}s_3^{n_3}-A^{n_1-n_2-n_3-k_1}s_1^{-n_1+k_1-1}s_2^{n_2-1}s_3^{n_3+1}\nonumber\\&\quad-A^{n_1-n_2-n_3-k_1}s_1^{-n_1+k_1-1}s_2^{n_2+1}s_3^{n_3-1}-A^{n_1-n_2-n_3-k_1+2}s_1^{-n_1+k_1-2}s_2^{n_2}s_3^{n_3}\nonumber\\&\quad-(-1)^{n_1}A^{-n_1-n_2-n_3-k_1-2}s_1^{k_1}s_2^{n_1+n_2}s_3^{n_1+n_3}-(-1)^{n_1}A^{-n_1-n_2-n_3-k_1}s_1^{k_1-1}s_2^{n_1+n_2-1}s_3^{n_1+n_3+1}\nonumber\\&\quad-(-1)^{n_1}A^{-n_1-n_2-n_3-k_1}s_1^{k_1-1}s_2^{n_1+n_2+1}s_3^{n_1+n_3-1}-(-1)^{n_1}A^{-n_1-n_2-n_3-k_1+2}s_1^{k_1-2}s_2^{n_1+n_2}s_3^{n_1+n_3}.\nonumber
            \end{align}
            Similarly,
            \begin{align}
                Right&=-A^{n_1-n_2-n_3-k_1-2}s_1^{-n_1+k_1}s_2^{n_2}s_3^{n_3}-A^{n_1-n_2-n_3-k_1}s_1^{-n_1+k_1-1}s_2^{n_2-1}s_3^{n_3+1}\nonumber\\&\quad-A^{n_1-n_2-n_3-k_1}s_1^{-n_1+k_1-1}s_2^{n_2+1}s_3^{n_3-1}-A^{n_1-n_2-n_3-k_1+2}s_1^{-n_1+k_1-2}s_2^{n_2}s_3^{n_3}\nonumber\\&\quad-(-1)^{n_1}A^{-n_1-n_2-n_3-k_1-2}s_1^{k_1}s_2^{n_1+n_2}s_3^{n_1+n_3}-(-1)^{n_1}A^{-n_1-n_2-n_3-k_1}s_1^{k_1-1}s_2^{n_1+n_2-1}s_3^{n_1+n_3+1}\nonumber\\&\quad-(-1)^{n_1}A^{-n_1-n_2-n_3-k_1}s_1^{k_1-1}s_2^{n_1+n_2+1}s_3^{n_1+n_3-1}-(-1)^{n_1}A^{-n_1-n_2-n_3-k_1+2}s_1^{k_1-2}s_2^{n_1+n_2}s_3^{n_1+n_3}\nonumber\\&=Left.\nonumber
            \end{align}
        \end{proof}
    \end{lemma}
    Therefore,
    \begin{align}
        Left&=R_{23}(n_1,n_2,n_3)\nonumber\\&\quad -A^{n_1-n_3}\sum_{i=0}^{n_1}(-1)^i\left(R_{12}(n_1-i,n_2-i,n_3+i)\right.\nonumber\\&\quad\left.+A^2R_{12}(n_1+1-i,n_2-1-i,n_3-1+i)\right)\nonumber\\&\quad-R_{23}(n_1,-n_2+k_1,-n_3+k_3)\nonumber\\&\quad -A^{n_1+n_2-k_2}\sum_{i=0}^{n_1}(-1)^i\left(R_{13}(n_1-i,n_2-k_2-2-i,-n_3+k_3-i)\right.\nonumber\\&\quad \left.+A^2R_{13}(n_1+1-i,n_2-k_2-1-i,n_3-k_3-1-i)\right)\nonumber\\&\quad +A^{n_1-n_3}\sum_{i=0}^{n_1}(-1)^i\left(R_{12}(-n_1+k_1+i,-n_2+k_2+i,n_3+i)\right.\nonumber\\&\quad\left.+A^2R_{12}(-n_1+k_1-1+i,-n_2+k_2+1+i,n_3-1+i)\right)\nonumber\\&\quad +A^{n_1+n_2-k_2}\sum_{i=0}^{n_1}(-1)^i\left(R_{13}(-n_1+k_1+i,n_2-k_2-2-i,n_3+i)\right.\nonumber\\&\quad \left.+A^2R_{13}(-n_1+k_1-1+i,n_2-k_2-1-i,n_3+1+i)\right)\nonumber\\&\xlongequal{Lemma\ \ref{Lemma 1_2}}\ 0.\nonumber
    \end{align}
\end{proof}

\putbib[main]
\end{bibunit}

\end{document}